\documentclass{amsart}
\usepackage[english]{babel}
\usepackage[utf8]{inputenc}
\usepackage{booktabs}
\usepackage{csquotes}
\usepackage{lmodern}
\usepackage{multirow}
\usepackage[T1]{fontenc}
\usepackage{tikz}
	\usetikzlibrary{calc}
	\usetikzlibrary{decorations.pathmorphing}
	\usetikzlibrary{decorations.shapes}
	\usetikzlibrary{arrows}
	\tikzset{%
		cross/.style = {preaction={draw=white,line width=4pt}},
		smalldot/.style = {draw,circle,inner sep=1pt,thick,fill},
		dot1/.style = {draw,circle,minimum size=3pt,inner sep=0pt,outer sep=3pt,thick,fill},
		dot2/.style = {draw,circle,minimum size=3pt,inner sep=0pt,outer sep=3pt,thick},
		dot3/.style = {draw,rectangle,minimum size=3pt,inner sep=0pt,outer sep=3pt,thick,fill},
		dot4/.style = {draw,rectangle,minimum size=3pt,inner sep=0pt,outer sep=3pt,thick},
		steiner1/.style = {},
		steiner2/.style = {dash pattern=on .2cm off .1cm,thick},
		steiner3/.style = {very thick,dotted},
		steiner4/.style = {decorate,decoration={crosses,segment length=.18cm}},
		steiner5/.style = {decorate,decoration={zigzag,amplitude=.04cm}},
		steiner6/.style = {decorate,decoration={shape backgrounds,shape=rectangle,shape size=.05cm,shape sep=.1cm}},
		region/.style = {shape=circle,draw,text width=3ex,align=center,inner sep=0pt},
	}
\usepackage{microtype}
\usepackage[subrefformat=parens,labelfont=rm]{subcaption}
\usepackage{amsmath}
\usepackage{amssymb}
\usepackage{amsthm}
	\providecommand{\R}{\mathbb R}
	\providecommand{\Z}{\mathbb Z}
	\theoremstyle{plain}
	\newtheorem{lemma}{Lemma}[section]
	\newtheorem{theorem}[lemma]{Theorem}
	\newtheorem{maintheorem}{Theorem}
	\newtheorem{corollary}[lemma]{Corollary}
	
	\newtheorem{proposition}[lemma]{Proposition}
	\theoremstyle{definition}
	\newtheorem{definition}[lemma]{Definition}
	\theoremstyle{remark}
	\newtheorem{remark}[lemma]{Remark}
	\DeclareMathOperator{\Star}{star}
	\DeclareMathOperator{\dist}{dist}
	\DeclareMathOperator{\sdist}{d}
	\DeclareMathOperator{\vol}{vol}
	\DeclareMathOperator{\area}{area}
\usepackage{enumitem}
	\setitemize{leftmargin=*}
	\setenumerate{label=\emph{(}\roman*\emph{)},leftmargin=*,widest*=3,ref=\textrm{(}\textit{\roman*}\textrm{)}}
	\newlist{caselist}{enumerate}{1}
	\setlist[caselist]{font=\rmfamily\mdseries\itshape,label=Case~\arabic*:,ref=Case~\arabic*,listparindent=0pt,labelwidth=0pt,itemindent=!,labelindent=0pt,leftmargin=0pt,align=left,parsep=2pt plus 2pt minus 1pt,itemsep=.25\baselineskip}
\usepackage{hyperref}
\usepackage{comment}
\graphicspath{{images/}{./}}
\newcommand*\dia{\textup{\textsf{dia}}}
\newcommand*\cds{\textup{\textsf{cds}}}
\newcommand*\sqp{\textup{\textsf{sqp}}}
\newcommand*\bnn{\textup{\textsf{bnn}}}
\newcommand*\pcu{\textup{\textsf{pcu}}}
\newcommand*\ths{\textup{\textsf{ths}}}
\newcommand*\srs{\textup{\textsf{srs}}}
\newcommand*\hcb{\textup{\textsf{hcb}}}
\newcommand*\sql{\textup{\textsf{sql}}}
\newcommand*\gsquare{\mathbin{\scalebox{.707}{$\square$}}}
\DeclareMathOperator{\rank}{rank}

\begin{document}

\title[Periodic networks of fixed degree minimizing length]{Periodic 
networks of fixed degree\\minimizing length}
\date{\today}
\author[Alex \& Grosse-Brauckmann]{Jerome Alex and Karsten Grosse-Brauckmann}
\address{Technische Universit\"at Darmstadt, Fachbereich Mathematik (AG 3),
    Schlossgartenstr.~7, 64289 Darmstadt, Germany}
\email{jalex, kgb@mathematik.tu-darmstadt.de}

\begin{abstract}
  We study networks in $\R^n$ which are
  periodic under a lattice of rank~$n$
  and have vertices of prescribed degree $d\ge 3$.
  We minimize the length of the quotient networks, 
  subject to the constraint that the fundamental domain has 
  $n$-dimensional volume~$1$.
  For $n=3$ and degree $3\leq d\leq 6$
  we determine the minimizing networks with the least number 
  of vertices in the quotient, while for $d\ge 7$ we state a length estimate.
  For general $n$, we determine the unique minimizers 
  with $d=n+1$ and $d=2n$.
\end{abstract}

\maketitle

\section{Introduction}

We use the term \emph{network} to denote a connected
graph with straight edges in Euclidean space~$\R^n$.
We assume the network is $n$-periodic, that is, 
invariant under some lattice~$\Lambda$ of rank~$n$, 
and that its quotient $N/\Lambda\subset\R^n/\Lambda$ is finite.
We are interested to minimize the length $L$ of this quotient,
without prescribing the lattice;
to set up a well-posed minimization problem we fix 
the volume $V$ of the fundamental domain $\R^n/\Lambda$.
Equivalently, we minimize the length quotient $L^n/V$.
In our previous work~\cite{periodicsteiner} we prove that for given 
dimension $n\ge 2$ the minimizers have $2n-2$ vertices, each of degree~$d=3$;
for Euclidean $3$-space we determine the \srs-network 
with the body centred cubic lattice as the unique minimizer.

For a natural system in Euclidean space,
material reasons can prescribe, however, a degree $d>3$ at the vertices.  
An example of a network with degree~$4$ is the well-known diamond network, 
see Figure~\ref{fig:network4}.
For simplicity, we consider here only the so-called $d$-regular case 
that $d$ agrees at all vertices.  It seems natural to ask:
\emph{What are the triply periodic networks in~$\R^3$ minimizing $L^3/V$
among networks with a prescribed degree $d$?}
We also ask:
\emph{How much larger is $L^3/V$ for $d\ge 4$ compared with
the case $d=3$?}
In the present paper we address these questions for networks 
whose quotient $N/\Lambda$ has the minimal number of vertices,
a case we call \emph{irreducible}.

We state answers to the questions in terms of
the graphs $B_k$, $D_{\ell,k}$, $D_k$ on one or two vertices
which are defined in Section~\ref{sec:topology}:
\begin{maintheorem}\label{thm:mainlow}
  Irreducible triply periodic networks $N\subset\R^3$ 
  with degree~$d\in \{4,5,6\}$ 
  can only have one of the following quotient graphs:
  $D_4$ or $D_{1,2}$ for $d=4$; $D_5$ or $D_{1,3}$ for $d=5$: and $B_3$ for $d=6$.
  The respective minimal values of $L/V^{1/3}$
  are quoted in Table~\ref{tab:summary}.
  The minimal network is unique (up to similarity)
  in each case except for $D_{1,2}$,
  where a one-parameter family minimizes.
  For $d\ge 7$ each irreducible network has a length quotient 
  still larger than all values for minimizers with $d=3$ to $6$
  quoted in Table~\ref{tab:summary}.
\end{maintheorem}
\noindent
See Theorems \ref{thm:diacds}, \ref{thm:bnn}, \ref{thm:sqp}, \ref{th:pcu}
for the precise statements for degree $d=4$ to $6$, 
and Corollary~\ref{cor:highdegthree} for the estimate for $d\ge 7$.

Our determination of the minimal networks in Euclidean space~$\R^3$ 
seems in agreement with the occurence of these networks in natural systems,
although the reasons leading to the networks in nature
are certainly more complex.
Indeed, as Table~\ref{tab:summary} indicates, 
the length quotient $L/V^{1/3}$ for the frequently encountered diamond network 
is only by 3\% larger than for the optimal Steiner network \srs. 
The two families \ths\ and \cds\ admit deformations into networks 
of smaller length, and so are less likely to occur.
Thus the next best candidate is a network which is also observed, 
namely \pcu\ with degree $d=6$, 
and a quotient by 12\% larger than~\srs.
There is a significant gap to the networks with $d=5$, 
which seem of minor physical importance, 
as their quotient $L/V^{1/3}$ is about 35\% larger compared to~\srs.

\begin{table}
	\centering
	\begin{tabular}{l*7c}
		\toprule
		degree & \multicolumn2c{$d=3$} & \multicolumn2c{$d=4$} & \multicolumn2c{$d=5$} & $d=6$\\
		\#vertices/\#edges\hspace*{-2em} & \multicolumn2c{$4$ / $6$} & \multicolumn2c{$2$ / $4$} & \multicolumn2c{$2$ / $5$} & $1$ / $3$\\
		\midrule
quotient graph & $K_4$ & $D_1\gsquare D_2$ & $D_4$ & $D_{1,2}$ & $D_{1,3}$ & $D_5$ & $B_3$ \\[.3mm]
		\multirow2*{minimizer} & \srs & \ths & \dia & \cds & \bnn & \sqp & \pcu \\[-.5mm]
		         & unique & family & unique & family & unique & unique & unique \\[.5mm]
		related surface & gyroid & --- & $D$ & $CLP$ & $H'$-$T$ & --- & $P$ \\
		\midrule
		$L/V^{1/3}$ & $\approx2.67$ & $\approx2.73$ & $\approx2.75$ & 3 & $\approx3.6$ & $\approx3.7$ & 3\\
	   & 100\% & 102.0\% & 102.9\% & 112.3\% & 134.8\% & 135.2\% & 112.3\%\\\bottomrule\\
	\end{tabular}
	\caption{Minimizing triply periodic networks with prescribed degree $3$ to~$6$
and their length quotients; all networks with the least possible number of vertices 
in the quotient are studied.  See text for acronyms of minimizers and minimal surfaces.
}
	\label{tab:summary}
\end{table}

The acronyms for the minimizing networks quoted in the fourth line of 
Table~\ref{tab:summary} are used by crystallographers, 
see~\cite{chemistrystructure} and also \href{http://www.rcsr.net}{rcsr.net}. 
We should note, however, that the lengths of our minimizers 
differ in some cases from the crystallographic standard representations
where edge lengths are chosen to coincide whenever possible.
Let us explain the acronyms.
In many cases they refer to a chemical compound: 
For the Steiner networks, \srs~stands for $\mathrm{SrSi_2}$ 
and \ths~for $\mathrm{ThSi_2}$. 
The diamond form of carbon explains \dia, 
and \cds~stands for $\mathrm{CdSO_4}$, 
while \bnn~denotes boron nitride nanotubes. 
Some other networks are named according to their lattice or geometry:
\pcu~denotes the primitive cubic unit, 
\sqp~denotes a network composed of square pyramids.
In the two-dimensional case, \sql\ relates to the square lattice
and~\hcb\ to the hexagonal or honeycomb network.
\medskip

While originally our interest was solely in the case of Euclidean space~$n=3$,
we have come to study higher dimensions as well.
One reason is that the case of general dimension indicates
which features are open to a systematic study, and which others
seem only accessible to a case-by-case study.
Another reason is that some of our techniques 
are natural to state in arbitrary dimension~$n$.  
They apply to the case that
the degree $d$ is larger than the space dimension:
\pagebreak[4]

\begin{maintheorem}\label{thm:mainhigh}
  Irreducible $n$-periodic networks $N$ with degree $d\geq n+1$ cover
  graphs with one or two vertices in the quotient.
  For $d=n+1$ and degree $d=2n$ we determine the unique minimizers,
  in particular we have sharp estimates for the quotient $L^n/V$, 
  see \eqref{eq:dipole} and~\eqref{eq:dlargertwon}.
\end{maintheorem}
\noindent
More specifically, the quotient graphs are described 
in Proposition~\ref{prop:structure}:
For even $d\ge 2n$, they have just one vertex, and so are unique;
the remaining cases with $d\ge n+1$ have two vertices and 
there are $\lfloor d/2 \rfloor$ possible topological types.

\begin{figure}
	\centering
	\begin{tikzpicture}[>=stealth']
		\draw[->](1.8,1.75)--(1.8,5.5) node[above]{$n$};
		\draw[->](2.4,1.5)--(9.5,1.5) node[right]{$d$};
		\foreach\x in {3,...,9} {\draw(\x,1.5)--+(0,-.2);}
		\foreach\y in {2,...,5} {\draw(1.8,\y)--+(-.2,0);}
		\foreach\x in {3,...,8} {\foreach\y in {2,...,4} {\draw[gray](\x-.05,\y)--(\x+.05,\y);\draw[gray](\x,\y-.05)--(\x,\y+.05);}}
		\foreach\x in {3,...,8} {\node[below] at (\x,1.3) {$\x$};}
		\foreach\y in {2,...,4} {\node[left] at (1.6,\y) {$\y$};}
		\draw[dashed](3,2)--(5.75,4.75);
		\node[above right,align=center]at(5.5,4.75){$d=n+1$:\\$n$-simplex};
  	\draw[dotted](3,2)--(3,4.4);
		\node[above,align=center]at(3,4.4){$d=3$:\\Steiner\\networks};
		\draw[dashed](4,2)--(8.5,4.25);
		\node[right,align=center] at (8.5,4.25){$d=2n$:\\$n$-cube};
		\draw(8,4)node[smalldot]{};
		\draw(5,4)node[smalldot]{};
		\draw(3,3)node[smalldot]{}node[left,align=left]{\srs\\[-.3\baselineskip]\ths};
		\draw(3,2)node[smalldot]{}node[left]{\hcb};
		\draw(4,2)node[smalldot]{}node[left]{\sql};
		\draw(4,3)node[smalldot]{}node[left,align=left]{\dia\\[-.3\baselineskip]\cds};
		\draw(5,3)node[smalldot]{}node[left,align=left]{\bnn\\[-.3\baselineskip]\sqp};
		\draw(6,3)node[smalldot]{}node[left]{\pcu};
    \node[region]  at (4.3,4.5) {I};
    \node[region]  at (5.8,3.6) {II};
    \node[region]  at (7.5,2.6) {III};
	\end{tikzpicture}
	\caption{Minimizers determined in the present paper are indicated with black dots;
for dimension $n=3$ these are described in Table~\ref{tab:summary}. 
The two dashed lines denote results for families of minimizers.
For networks within regions II and III (but not for those in~I) 
we have results on the network topology and 
estimates of the length quotient.}
	\label{fig:casediagram}
\end{figure}
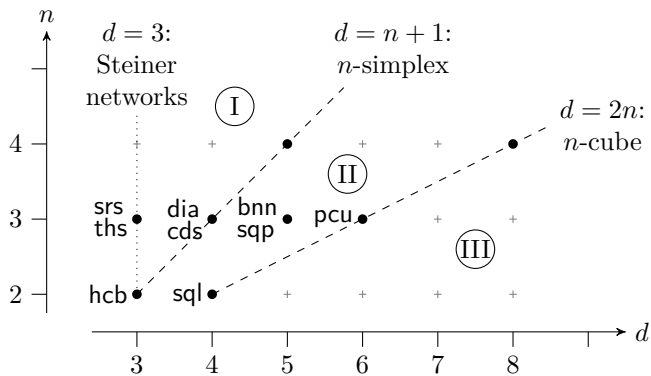

Figure~\ref{fig:casediagram} shows how the results for general $n$
relate to the ones for $n=3$. 
The minimizing simplicial networks with $d=n+1$ 
generalize the diamond or the hexagonal planar network 
to arbitrary dimension, and the primitive cubical networks
minimizing for $d=2n$ generalize 
the primitive network in $3$-space or the planar square lattice.

For the region with $d\leq n$ marked with I in Figure~\ref{fig:casediagram},
the network quotients have more than~$2$ vertices,
see~\eqref{eq:numbervertices}.
As shown in \cite{periodicsteiner},
for the Steiner case $d=3$ the quotient graph has at least $2n-2$ vertices.
It is known that the number of topologically different graphs with 
$2n-2$ vertices rapidly increases with~$n$. 
Thus we do not expect a good systematic theory for the case $d\leq n$.

While we do not offer a characterization of the minimizers 
with $d$ between $n+2$ and $2n-1$
corresponding to the region II of the figure, 
Theorem~\ref{thm:highdeg} implies that for each such~$n$ the corresponding 
minimizers have a length larger than the simplicial networks 
generalizing the diamond. 
Similar remarks apply to region~III with $d>2n+1$:
Here the primitive cubic network gives rise to a lower estimate,
see Theorem~\ref{th:pcu}.  That is, in regions II and~III
of the figure, the length quotient is estimated strictly by the 
minimizers represented by the dashed lines to their left.
\medskip
\pagebreak[4]

Having explained our main results, let us describe how
the paper is organized. After making precise our notation 
in Section~\ref{sec:defs} we identify the possible topologies 
of minimizers in Section~\ref{sec:topology}.
The simplex case $d=n+1$ and the primitive cubic case $d=2n$
are dealt with in Section~\ref{sec:simplex} and~\ref{sec:cube}, respectively.
The remaining two sections serve to complete the picture 
for Euclidean space $\R^3$:  In Section \ref{sec:degfour} 
we determine the minimizers for the two topological types 
which arise for irreducible networks with $d=4$, 
and in Section~\ref{sec:degfive}.
we analyse the much more involved case $d=5$ with its two
different irreducible topologies.
\medskip

We would like to comment on the significance of networks to surface theory
which motivated the present study.
Triply periodic embedded minimal surfaces were first constructed by Schwarz
and his students.
In 1970, Alan Schoen used networks, which he called skeletal graphs, 
in order to suggest further candidates for such surfaces; 
rigorous existence proofs were later obtained by Karcher~\cite{Karcher1989}.
Networks and oriented minimal surfaces have the same symmetry group,
and the connected component to the side of the minimal surface 
containing the network can be retracted to it, subject to the symmetry group.
In fact there are two networks, to either side of the embedded surface;
the networks are not necessarily congruent.
In Table~\ref{tab:summary} we include the relevant 
minimal surfaces in Schoen's terminology, 
namely the Schwarz $P$, $D$, and $CLP$ surface, as well as
Schoen's gyroid and $H'-T$.  We are not aware of minimal 
surfaces with the remaining two graphs \ths\ and \sqp.
We would like to add that also 
constant mean curvature surfaces were constructed 
in terms of networks by Kapouleas~\cite{Kapouleas1990} 
and recently by Traizet~\cite{Traizet2017}.

Nevertheless, in general there is no well-defined relationship between
such triply periodic surfaces and networks.
An attempt to define graphs for arbitrary minimal or
constant mean curvature Alexandrov embedded surfaces 
(not necessarily periodic) is due to Kusner~\cite{Kusner1991}:
He defines straight lines in terms of loop integrals 
which are well-defined on the first homology of the surface.
However, only in symmetric cases will these lines meet 
in vertices and thereby define edges of a network.

Numerical experiments made us aware of another possible approach to
produce the networks, at least for very symmetric cases~\cite{kgbroyalsociety}.
Suppose that for fixed lattice~$\Lambda$ and constant $C>0$
there is a periodic embedded surface $\Sigma\subset\R^3$
minimizing the Willmore energy $\int H^2\,dA$ in $\R^3/\Lambda$,
under the constraint that $\Sigma$ bounds a component $\Omega\subset\R^3$
with enclosed volume $C:=V(\Omega/\Lambda)>0$.  
Typically, $\Sigma$ is a triply periodic minimal surface, 
and due to symmetries the volume $C$ is half the volume of 
the fundamental domain.
Experiments with Brakke's Surface Evolver indicate
that a continuous deformation family $c\mapsto\Sigma_c$ exists for
$c\in (0,C]$, with initial surface $\Sigma_C=\Sigma$
and $c$ the volume of the component $\Omega_c/\Lambda$ deforming~$\Omega/\Lambda$.
In many, but not all of the cases we investigated, a network~$N$
arises as the singular limit $\lim_{c\to 0}\Sigma_c$ 
w.r.t.\ Hausdorff distance; the Willmore energy tends to infinity. %
Geometrically, the surfaces $\Sigma_c$ can be described as 
thin cylindrical tubes around the network.
\medskip

Let us conclude the introduction by mentioning some open problems.
We conjecture that minimizers for given $n,d$ are always irreducible.
It would be interesting to gather information about the networks with $d\leq n$,
corresponding to region I in Figure~\ref{fig:casediagram}.
Also, we would expect that for given $n$ the length quotient
$L^n/V$ is monotone in $d\geq n$; 
we know this holds when restricted to $d\geq2n$ even, see 
Remark~\ref{rem:evenmonotone}.
Finally, our assumption that the degree coincides at all vertices could be relaxed.  
\pagebreak[4]

\section{Periodic networks of fixed degree}\label{sec:defs}

We define networks for fixed degree similar to \cite{periodicsteiner}.
\begin{definition}
	An \emph{$n$-periodic network $N$ of degree $d$} is a connected simple graph, 
  immersed with straight edges of positive length into $\R^n$, 
  where $n\ge 2$, subject to the following conditions:
	\begin{itemize}
		\item All vertices have the same degree $d\ge 3$.
		\item $N$ is invariant under the action of a lattice $\Lambda$ of rank $n$.
		\item The quotient $\Gamma:=N/\Lambda$ is a finite graph, possibly with loops and multiple edges.
	\end{itemize}
	We call $V = V(\R^n/\Lambda)$ 
  the \emph{(spanned) volume} of $N$ and $L = L(N/\Lambda)$ its \emph{length}.
\end{definition}

Recall that a \emph{lattice} of rank $n$ is a set 
$\Lambda=\{\sum_{i=1}^n a_ig_i : a_i\in\mathbb Z\}\subset\R^n$, 
where the vectors $g_1,\dots,g_n\in\R^n$ are linearly independent.
A network is \emph{immersed} if the star of each vertex is embedded.
Here the \emph{star} of a vertex~$p$, denoted~$\Star p$,
is the union of the edges from~$p$ to its incident vertices.
Clearly, the immersion condition implies simplicity of the network.

If an abstract finite graph~$\Gamma$ with vertices of degree~$d$ is given, 
then our networks can be described as immersions
of certain abelian coverings of~$\Gamma$; 
see Sunada~\cite{sunada2012topological} for the covering theory of graphs. 
Note, however, that an immersed network can have a non-immersed quotient: 
For instance, in $\R^2$ we consider the network~$N$ of degree~$6$ 
which is the $\Z^2$-orbit of the edges from the origin 
to $(0,1)$, $(1,0)$ and~$(2,1)$.  Then the star of a vertex
has a self-intersection when taken in $\R^n/\Z^2$, but not in $\R^n$.

We are interested in networks which are optimal in the sense 
that the length $L$ of the quotient network $N/\Lambda$ is minimal. 
As in \cite{periodicsteiner} we minimize~$L$ 
subject to the constraint that the $n$-dimensional volume $V=V(\R^n/\Lambda)$ 
of a fundamental domain is fixed to~$1$. 
Note that the space of lattices subject to this constraint is non-compact.
Equivalently, we minimize the scaling-invariant \emph{length quotient}~$L^n/V$.

Let us state a well-known necessary condition for a network to be a minimizer. 
Consider a vertex~$p$ of a network~$N$ with $d$~neighbours $\{q_1,\dots,q_d\}$.
\begin{definition}
  The \emph{total force at~$p$} of a network $N$, 
  exerted by the 
  $d$ edges incident to $p$, is the vector
\begin{align}
	F(p):=\sum_{i=1}^d\frac{p-q_i}{\vert p-q_i\vert}\,.\label{eq:equilibrium}
\end{align}
	A network $N$ is \emph{balanced} if and only if $F(p)=0$ holds at each vertex $p\in N$.
\end{definition}

\noindent  
The length of the star gives rise to the convex function
$
  p\mapsto L(\Star p) := \sum_{i=1}^d\vert p - q_i\vert
$,
and so for given $q_i$ the function $L$ attains a minimum at a unique 
critical point~$p$.
We say $p$ is \emph{critical} for $L(\Star p)$ if 
\begin{align}
 0=\frac{\mathrm d}{\mathrm dt}L\bigl(\Star (p+tv)\bigr)\Big|_{t=0} 
	= \Bigl\langle\sum_{i=1}^d\frac{p-q_i}{\vert p-q_i\vert},v\Bigr\rangle
  \qquad\text{for all }v\in\R^n. 
\end{align}
Clearly length criticality is equivalent to force balancing:
\begin{proposition}\label{prop:equilibrium}
  A network $N$ is balanced if and only if each vertex $p\in N$
  is critical for $L(\Star p)$.
\end{proposition}
\noindent 
In particular, all vertices of a minimizer for $L^n/V$ are balanced.
We want to analyse networks with the simplest topology:
\begin{definition}
We call an $n$-periodic network $N$ of degree~$d$ \emph{irreducible} if its quotient $N/\Lambda$ has the least number of vertices possible for a balanced network of degree~$d$ in~$\R^n$.
\end{definition}
Irreducibility can be related to the \emph{circuit rank} 
of the connected graph~$N/\Lambda$, 
\[ \rank N := 1 - \#\text{vertices in $(N/\Lambda)$}+\#\text{edges in $(N/\Lambda)$}\,.\]
For connected finite graphs the rank is precisely the number of generators 
of $H_1(N/\Lambda,\mathbb Z)$, and so for $N$ to be $n$-periodic 
we must have $\rank N\ge n$.
To see this, consider a spanning tree $T\subset N/\Lambda$ 
of the quotient graph.
Then $H_1(T,\mathbb Z)$ is trivial, and reinserting edges one by one increases the rank of~$T$ as well as the number of generating cycles in~$T$ by $1$ each. 

We want to classify the topology of irreducible networks.
An $n$-periodic network $N$ of degree $d$ must have
$2\cdot\# \text{edges} = d\cdot\#\text{vertices}$ and so
$\rank N = 1 + \bigl(\frac d2-1\bigr)\cdot\#\text{vertices}$.
Therefore an $n$-periodic network of degree $d$ satisfies
\begin{equation}\label{eq:numbervertices}
  \#\text{vertices} \ge \frac{2n-2}{d-2};
\end{equation}
in particular, a network with degree $d< 2n$ has at least two vertices,
and a network with $d<n+1$ at least three.

\begin{remark}
  For the Steiner case, $d=3$, an $n$-periodic network $N$ is irreducible 
  if and only if $\rank N = n$.   
  Indeed, a balanced network with one vertex and $n$~loops in the quotient
  can be split into a balanced Steiner network on $2n-2$~vertices, 
  see~\cite[Theorem 2.3]{periodicsteiner}.
	For $d>3$, however, $\rank N$ can be larger than~$n$:
  Proposition~\ref{prop:structure} below gives an irreducible
  example with $d=5$, $n=3$ on $2$~vertices, so that the rank is~$4$.
\end{remark}

\section{Topology of irreducible \texorpdfstring{$n$}{n}-periodic networks}
\label{sec:topology}

One or two vertices are clearly the simplest case 
for the topology of a quotient graph.
Our goal is to show that no more vertices are needed
for networks of degree $d>n$ to be irreducible.
We start by introducing connected multigraphs with one or two vertices,
see~Figure~\ref{fig:structure}:

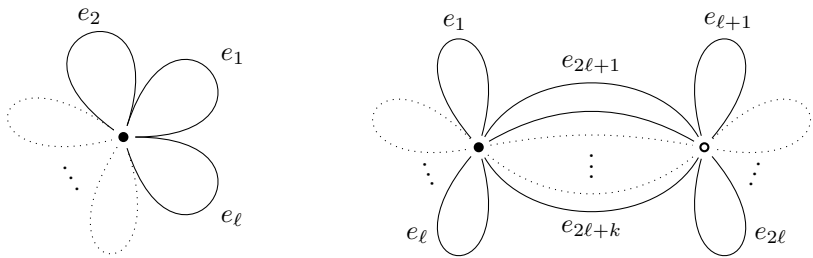
\begin{figure}[b]
	\centering
	\begin{subfigure}[t]{.38\linewidth}
		\centering
		\begin{tikzpicture}[scale=3]
			\coordinate[dot1](p) at (0,0);
			\draw(p) to [out=0,in=72,looseness=30] node[above right]{$e_1$} (p);
			\draw(p) to [out=72,in=144,looseness=30] node[above]{$e_2$}(p);
			\draw[dotted](p) to [out=144,in=194,looseness=40] (p);
			\draw[dotted](p) to [out=238,in=288,looseness=40] (p);
			\node[rotate=-60] at (220:.3) {$\cdots$};
			\draw(p) to [out=288,in=0,looseness=30] node[below right]{$e_{\ell}$} (p);
		\end{tikzpicture}
	\end{subfigure}
	\begin{subfigure}[t]{.6\linewidth}
		\begin{tikzpicture}[scale=3]
			\coordinate[dot1](p) at (0,0);
			\coordinate[dot2](q) at (1,0);
			\draw(p) to [out=60,in=120] node[above]{$e_{2\ell+1}$}  (q);
			\draw(p) to [out=30,in=150] (q);
			\node at (.5,-.05) {$\vdots$};
			\draw[dotted](p) to [out=10,in=170] (q);
			\draw[dotted](p) to [out=-40,in=220] (q);
			\draw(p) to [out=-60,in=240] node[below]{$e_{2\ell+k}$}(q);
			\draw(p) to [out=130,in=75,looseness=35] node[above]{$e_1$} (p);
			\draw[dotted](p) to [out=135,in=190,looseness=35] (p);
			\node[rotate= 20] at ($(p)+(200:.25)$) {$\vdots$};
			\draw(p) to [out=-75,in=-130,looseness=35] node[left,pos=.7]{$e_\ell$}(p);
			\draw(q) to [out=50,in=105,looseness=35] node[above]{$e_{\ell+1}$} (q);
			\draw[dotted](q) to [out=45,in=-10,looseness=35] (q);
			\node[rotate=-20] at ($(q)+(-20:.25)$) {$\vdots$};
			\draw(q) to [out=-105,in=-50,looseness=35] node[right,pos=.7]{$e_{2\ell}$} (q);
		\end{tikzpicture}
	\end{subfigure}\vspace*{-.7cm}
	\caption{The bouquet graph $B_\ell$ (left) and 
the double bouquet graph $D_{\ell,k}$ (right).
The latter is obtained by joining two copies of $B_\ell$ with $k$~edges.
As stated in Proposition~\ref{prop:structure},
irreducible networks of degree~$d>n$ cover one of these graphs.}
	\label{fig:structure}
\end{figure}
\begin{itemize}
	\item The \emph{bouquet graph} $B_\ell$ of degree~$2\ell$
    consists of one vertex with $\ell\geq0$ loops.
	\item The \emph{double bouquet graph} $D_{\ell,k}$
   has degree~$2\ell+k$ and consists of the union of two bouquet graphs 
   $B_\ell$ with $\ell\ge 0$, connected with $k\geq1$ edges. 
   Specifically, we call $D_k:=D_{0,k}$ the \emph{dipole graph} of degree~$k$.
\end{itemize}

We begin with an existence statement.

\begin{lemma}\label{lemma:structure}
	Let $n\geq2$ and $\Lambda\subset\R^n$ be a lattice.
  There exist balanced $n$-periodic networks of degree~$d$
	\begin{enumerate}
		\item\label{item:structureclaimtwovertices} for~$d\geq n+1$ 
      such that the quotient is a double bouquet graph, and moreover
		\item\label{item:structureclaimonevertex} for even $d\geq2n$ 
      such that the quotient is the bouquet graph $B_{d/2}$.
	\end{enumerate}
\end{lemma}
\begin{proof}
	We distinguish three cases to construct the graphs; compare with Figure~\ref{fig:structurecases}.
	\begin{caselist}
	  \item[\emph{(}ii\emph{)}:]
Suppose $d$ is even and $d\geq2n$. 
To define $N$ pick first a point~$p\in\R^n$ and connect it 
with points of $p+(\Lambda\setminus\{0\})$ with edges as follows.
Choose $n$~vectors generating the lattice~$\Lambda$, 
and use them to define a set of $n\leq d/2$ edges.
Supplement this edge set in a way that the resulting set of $d/2$ edges 
does not contain any pair of opposite edges.
Then take the $\Lambda$-orbit of this edge set 
to define a network~$N$ of degree~$d$, which is balanced and has rank~$n$; 
moreover, the star of~$p$ is embedded, implying that $N$ is immersed.
Observe the quotient graph $N/\Lambda$ is topologically~$B_{d/2}$.%
\begin{figure}[b]
	\centering
	\begin{subfigure}{.3\linewidth}\centering
		\includegraphics[width=.8\linewidth]{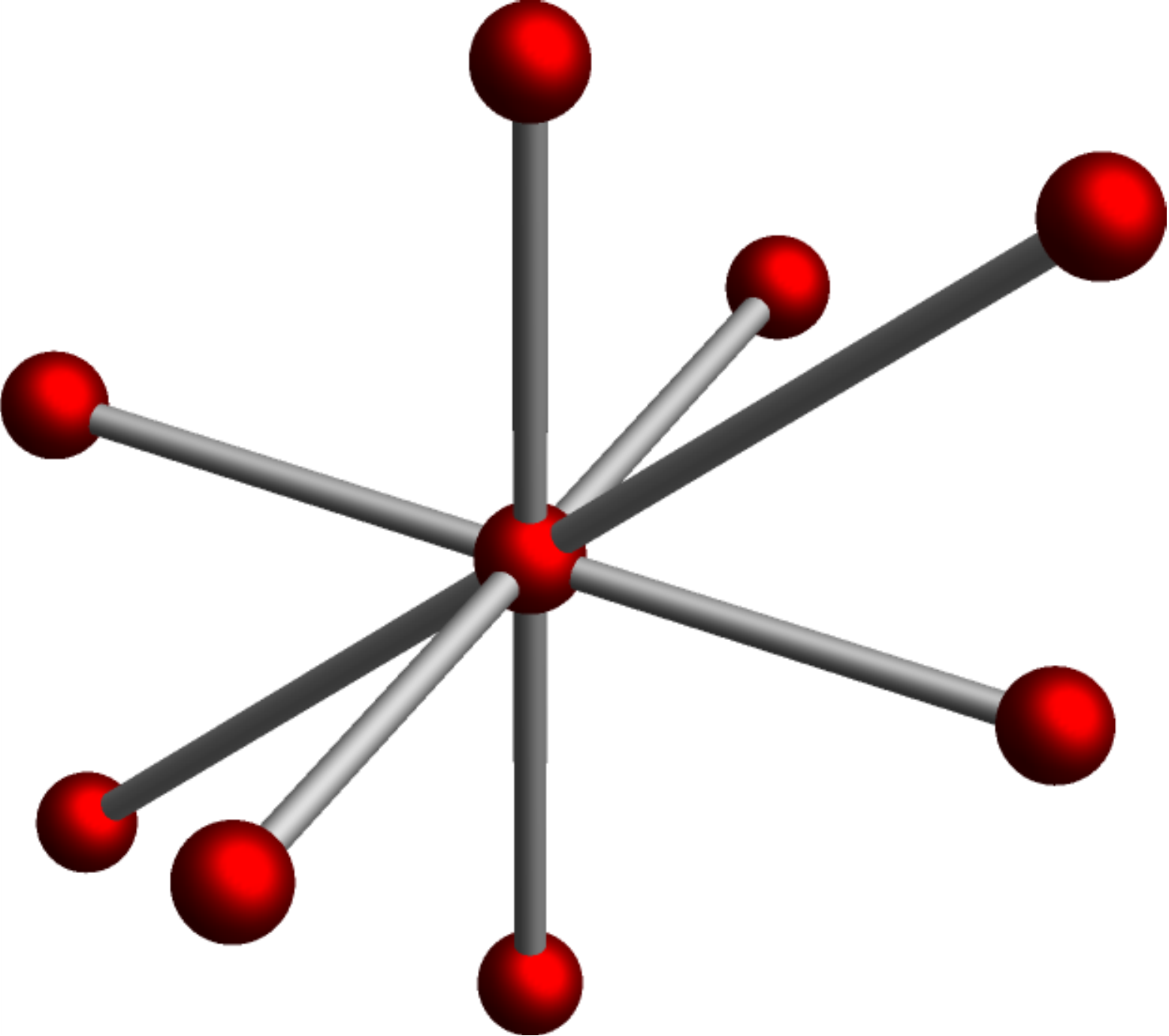}
	\end{subfigure}
	\begin{subfigure}{.29\linewidth}\centering
		\includegraphics[width=.7\linewidth]{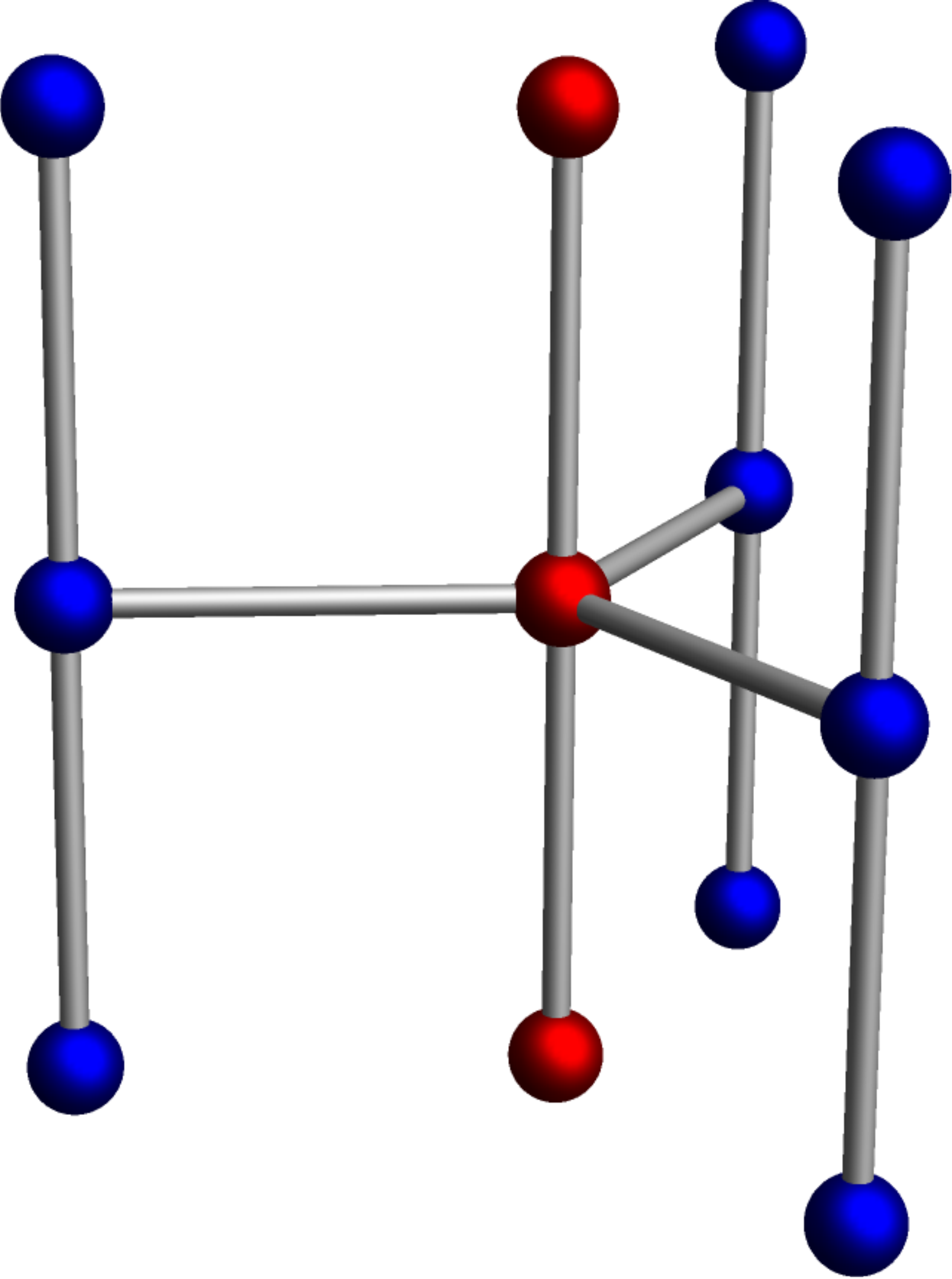}
	\end{subfigure}
	\begin{subfigure}{.34\linewidth}\centering
		\includegraphics[width=\linewidth]{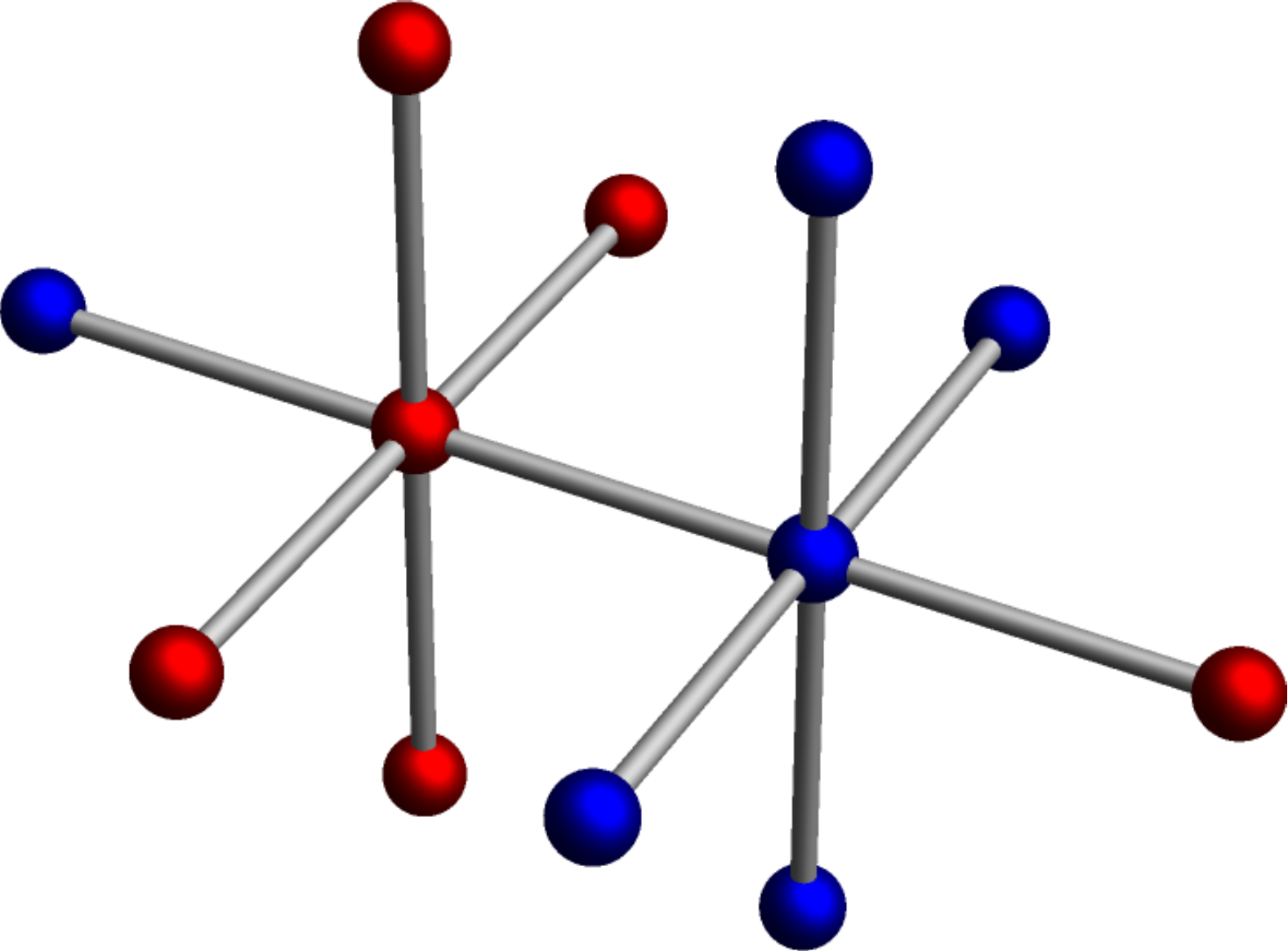}
	\end{subfigure}
	\caption{Construction of balanced $n$-periodic networks with prescribed 
  degree.  The figures correspond to the three cases in the proof 
  of Lemma~\ref{lemma:structure} for $n=3$:
  They show degree $d=8$, $d=5$, $d=6$
  and give rise to quotient graphs $B_4$, $D_{1,3}$, $D_{2,2}$, respectively.
}
	\label{fig:structurecases}
\end{figure}
	\item[\emph{(}i\emph{)}, case $d\geq n+1$ odd:]
We construct a network of degree~$d$ with quotient $D_{(d-3)/2,\, 3}$.
The network turns out to be a generalization of the \bnn\ network,
see Figure~\ref{fig:network5}.
Suppose $\Lambda$ is generated by $g_1,\ldots,g_n$.
Let $P$ be the plane in $\R^n$ spanned by $g_1$ and~$g_2$.

In a first step we
construct a balanced network of degree~$3$ in $P$ with quotient $D_3$.
We may assume the two generators $g_1,g_2$ 
of $\Lambda_0:=\Lambda\cap P$ are chosen to enclose an angle in $[\pi/3,\pi/2]$.
Then the triangle with vertices $0,g_1,g_2$ 
contains a Fermat point~$q$ in its interior, and so
the three edges connecting $q$ to $0$, $g_1$, $g_2$ are balanced at~$q$.
Let $N_2$ be the $\Lambda_0$-orbit of this tripod, 
which is an embedded network of degree~$3$ with topology~$D_3$.
Note that $N_2$ is balanced, since a network with two vertices 
in the quotient is balanced on both vertices if it is balanced at one vertex.

The second step is a product construction similar to the above 
proof of part \ref{item:structureclaimonevertex}.
Connect the vertex $0\in\Lambda$ with $(d-3)/2$ vertices in 
$\Lambda\setminus\Lambda_0$, and also the other vertex~$q$ 
with $(d-3)/2$ vertices in $q+(\Lambda\setminus\Lambda_0)$,
such that the resulting $d-3$ edges extend the 
set of $n-2\le d-3$ vectors $g_3,\ldots,g_n$, 
and such that neither at $0$ nor at~$q$ there is a pair of opposite edges.
The $\Lambda$-orbit $N$ of these edges
is an immersed network of degree~$d$.  
Note that $N$ is connected, as each lattice generator is represented by an edge.
All edges incident to a point either have direction in $P$ or 
occur in opposite pairs and so $N$ is balanced.
	\item[\emph{(}i\emph{)}, case $d\geq n+1$ even:] 
We construct a network $N$ of degree~$d$ with quotient $D_{(d/2)-1,\,2}$. 
Pick a generator $g_1$ of~$\Lambda$ and consider the edge $e$ from $0$ to~$g_1$.
Then choose $q$ in the interior of~$e$ 
and proceed as in the second step of the odd case: 
Connect each of $0,q$ to a point in $\Lambda\setminus\Z g_1$ 
with $\frac d2-1$ edges, this time including 
the directions of the remaining $n-1$ lattice vectors 
into the total edge set (thereby using the assumption $d-2\ge n-1$).
Again the $\Lambda$-orbit $N$ of this edge set satisfies all requirements.
Note that this construction agrees with the one for \ref{item:structureclaimonevertex}
if $q$ were~$0$.
	\end{caselist}%
\vspace{-1mm}

\end{proof}

If an $n$-periodic network in $\R^n$ has degree~$3$ 
it must have a quotient with at least $2n-2$ vertices.
Thus the topology becomes more complex with increasing 
dimension~$n$.
In contrast, for sufficiently high degree $d$ the lemma implies
that irreducible networks have a simple topology:
\begin{proposition}\label{prop:structure}
	Let $N$ be an irreducible $n$-periodic network of degree $d\geq n+1$.
	If $d$ is even and $d\geq2n$, then $N$ covers $B_{d/2}$.
	For all other $d\geq n+1$, the network~$N$ covers one of the 
  graphs $D_{\ell,k}$ where $2\leq k\leq d$ and 
  $d-k=2\ell$.
\end{proposition}
\begin{proof}
  For even $d\geq2n$, 
  Lemma~\ref{lemma:structure}\,\ref{item:structureclaimonevertex} asserts the
  existence of a network $N$ whose quotient $B_{d/2}$ has one vertex; 
  clearly $N$ is irreducible.  
  A finite graph with one vertex necessarily has even degree.
  For odd $d\geq n+1$, networks with two vertices exist 
  by part~\ref{item:structureclaimtwovertices} of the lemma, 
  and so these networks are irreducible.
  
  It remains to show that for even $d$ with 
  $n+1\leq d < 2n$ networks with two vertices are irreducible.
  On the contrary, suppose the quotient has only one vertex, 
  i.e., it is $B_{d/2}$.
  Since the quotient graph of an $n$-periodic network has 
  circuit rank at least~$n$, this gives $n\ge \rank B_{d/2}=d/2$, 
  ruling out this case.

  Finally a quotient $D_{\ell,k}$ with $k=1$ is impossible, 
  as an immersion covering $D_{\ell,1}$ cannot be balanced. 
\end{proof}
\begin{remark}
  For $d$ odd the number of graphs which are admissible 
  for the Proposition is $\lfloor d/2 \rfloor$ and so increases with~$d$.
  We should expect that minimizers favour
  a small number of loops $\ell$, 
  since it seems easier to make the $k$ bridges short.
  Nevertheless we will see that for $n=3$ and $d=5$ 
  the quotient graph $D_{1,3}$ leads to a shorter minimizer
  than~$D_5$.
\end{remark}

\section{Networks of degree \texorpdfstring{$d=n+1$}{d=n+1}}\label{sec:simplex}

We want to determine optimal $n$-periodic networks of degree $n+1$. 
For dimension $n=3$ the degree is $d=4$, 
and the minimizer 
will turn out to be the well-known diamond network, 
which can be characterized by the fact
that the neighbours of each vertex form the vertices of a regular tetrahedron.

In the present section we obtain the same characterization in arbitrary
dimension: 
The minimizers among irreducible $n$-periodic networks of degree $d=n+1$ 
are networks $N$ for which each vertex~$q\in N$ is the center 
of symmetry of a regular $n$-simplex, defined by the neighbours of~$q$. 
This will be shown in Theorem~\ref{thm:dipole}.

Our first goal is an estimate on the length for a graph $G_0$
connecting the origin to the vertices of an arbitrary simplex~$\Delta$:
\begin{proposition}\label{prop:nsimplex}
	Let $\Delta$ be an $n$-simplex with vertices $p_0,\ldots,p_n\in\R^n$ 
  and volume $V(\Delta)>0$. Then
	\begin{align}\label{eq:nsimplexformula}
		\frac{(L(G_0))^n}{V(\Delta)}\geq n!\sqrt{(n+1)^{n-1}\,n^n}\,,
	\end{align}%
	where we set $L(G_0):=\sum_{i=0}^n|p_i|$. 
  Equality holds if and only if $\Delta$ is a regular $n$-simplex
  with symmetry centre the origin.
\end{proposition}
Our proof depends on the estimate contained in the next lemma.  
Since we intend to use the estimate also 
for the proof of Theorem~\ref{thm:sqp} below, 
we state it for a case more general than a simplex,
namely for a pyramid.  

Consider a convex polyhedron~$E$ 
contained in the hyperplane $P:=\R^{n-1}\times\{0\}$,
such that $E$ has $k\ge n\ge 2$ pairwise distinct vertices 
$p_1,\ldots,p_k\in P$.
We assume $E$ has positive $(n-1)$-dimensional volume~$V_E>0$.
We then take a pyramid~$\Delta\subset\R^n$ with base~$E$ 
and apex $p_0\in\R^n\setminus P$ as in Figure~\ref{fig:pyramid}. 
We denote with $V(\Delta)>0$ its $n$-dimensional volume.
Moreover, we consider an arbitrary point $q\in\R^n$ 
and a graph $G_q$ which is the union of the edges from~$q$
to the vertices $p_0,\ldots,p_k$ of~$\Delta$. 
We denote its length by~$L(G_q)$, and the total length 
of the edges from~$q$ to the base vertices by 
$s:=\sum_{i=1}^k\vert p_i-q\vert>0$.

\begin{figure}
		\centering
		\begin{tikzpicture}[ultra thick,scale=.6]
			\coordinate(p1) at (0,-2);
			\coordinate(p2) at (-.4,0);
			\coordinate(p3) at (5.5,0);
			\coordinate(p4) at (5,-2);
			\coordinate(p5) at (1,5);
			\coordinate(q0) at (1,1.5);
			\draw[ultra thick](q0) -- (p3);
			\draw[gray,thick,cross](p4) -- (p5);
			\draw[gray,thick,dashed](p2) -- (p3);
			\draw[gray,thick](p3) -- (p4);
			\draw[ultra thick](q0) -- (p2);
			\draw[gray,thick,cross](p1) -- (p5);
			\draw[gray,thick](p1) -- (p2);
			\draw[ultra thick](q0) -- (p1);
			\draw[ultra thick](q0) -- (p4);
			\draw[ultra thick](q0) -- (p5);
			\draw[gray,thick](p2) -- (p5);
			\draw[gray,thick](p3) -- (p5);
			\draw[gray,thick](p1) -- (p4);
			\node[left] at (p2) {$p_2$};
			\node[above] at (p5) {$p_0$};
			\node[right] at (p3) {$p_3$};
			\node[right] at (p4) {$p_4$};
			\node[left] at (p1) {$p_1$};
			\node[above right] at (q0) {$q$};
			\node at ($.5*(p1)+.5*(p3)$) {$E$};
		\end{tikzpicture}
		\caption{%
The pyramid $\Delta$ of Lemma~\ref{lemma:pyramid} 
with base $E$ and apex~$p_0$. The graph $G_q$ connects the vertices 
of $\Delta$ with a further point~$q$.}
		\label{fig:pyramid}
\end{figure}
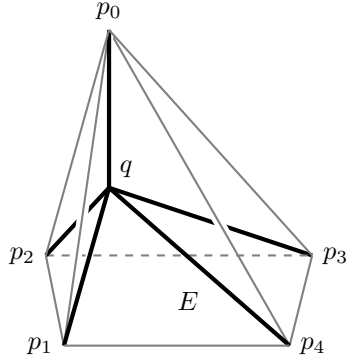

\begin{lemma}\label{lemma:pyramid}
  For given $p_0,p_1,\ldots,p_k$ and each $q\in\R^n$ the length of $G_q$ satisfies
	\begin{align}\label{eq:pyramid}
  \frac{(L(G_q))^n}{V(\Delta)}
  \geq\frac{n^2}{V_E}\biggl(\frac{k^2-1}{k^2}\, \frac{ns}{n-1}\biggr)^{n-1}.
	\end{align}
  The equality case is equivalent to the following conditions:
	$p_0-q$ is perpendicular to~$P$, as well as
  \begin{align}\label{eq:pyramidequality}
		\vert p_1-q\vert = \dots = \vert p_k-q\vert 
    = \frac{k(n-1)}{k^2-n}\vert p_0-q\vert\,,\quad\text{and}\quad
    \dist(q,P)=\frac s{k^2}\,.
  \end{align}
\end{lemma}
\begin{proof}
	Set $x_i := \vert p_i-q\vert$ for $i=1,\ldots,k$, 
        and $z:=\vert p_0-q\vert$. Then
	\[L(G_q) = \sum_{i=1}^kx_i+z = s+z\,,\]
  which is positive due to $s > 0$.
	Setting $h:=\dist(q,P)\ge 0$ 
  we can estimate the volume of the pyramid $\Delta$ by
	\begin{align}\label{eq:pyramidvolume}
		V(\Delta) \leq \frac1n(h+z)V_E\,,
	\end{align}
	where equality is attained if and only if $p_0-q$ is perpendicular to~$P$ 
  and $q$ lies in the closed slab of $\R^n$ between $P$ and $p_0$.
  Note that $p_0\notin P$ implies $h+z>0$.  
  Therefore, an equivalent inequality is 
	\begin{align}\label{eq:pyramidfunctionz}
		\frac{(L(G_q))^n}{V(\Delta)}\geq\frac{n(s+z)^n}{V_E(h+z)}\,.
	\end{align}
	For a moment, let us regard the right-hand side of \eqref{eq:pyramidfunctionz} 
  as a function of $z\in (-h,\infty)$; due to $h<s$ this function is positive.
  Differentiation yields the unique critical point 
	\[
    z_0:=\frac{s-nh}{n-1}\,.
  \]
  As $z$ tends to $-h$ or to infinity, 
  the right hand side of \eqref{eq:pyramidfunctionz} 
  tends to infinity, and so $z_0$ assigns a minimum 
  to the right hand side of~\eqref{eq:pyramidfunctionz}.
  But $s > kh\ge nh$ implies $z_0>0$ so that
  we have shown that for $z\in(0,\infty)$ 
  the right-hand side of \eqref{eq:pyramidfunctionz} 
  takes a unique strict minimum at $z_0$.
  
  Inserting $z_0$ into the inequality \eqref{eq:pyramidfunctionz} 
  yields
	\begin{align}\label{eq:pyramidfunctionh}
		\frac{(L(G_q))^n}{V(\Delta)}
    \geq\frac{n^2}{V_E}\Big(\frac{n(s-h)}{n-1}\Big)^{n-1}\,.
	\end{align}
  In particular, \eqref{eq:pyramid} holds strictly in case $h=0$,
  implying the lemma for this case.
  Thus we may assume $h>0$ in the following.

  The existence of $z_0$ means that there exists a $q\in\R^n$ 
  minimizing $(L(G_q))^n/V(\Delta)$ for the given~$p_i$'s.
  The equality discussion for~\eqref{eq:pyramidfunctionz} implies we must
  have $p_0-q\in P^\perp$, and since \eqref{eq:pyramidfunctionz} has a
  strict mimimum at $z_0>0$, we have $|p_0-q|=z_0>0$, 
  and so $q\not=p_0$, in particular.

  The volume $V(\Delta)$ is independent of $q$, 
  so that $q$ also minimizes $L(G_q)$.  
  Since all edge lengths are positive, $G_q$ must be balanced at~$q$.
  The balancing formula \eqref{eq:equilibrium} gives $\sum_{i=1}^k h/{x_i}=1$.
  This harmonic mean can be estimated by an arithmetic mean,
	\begin{align}\label{eq:pyramidheight}
		h = \Big(\sum_{i=1}^k\frac1{x_i}\Big)^{-1}\leq \frac s{k^2}\,,
	\end{align}
	where equality 
  holds if and only if $x_1=\dots=x_k$. 
  Combining \eqref{eq:pyramidfunctionh} and \eqref{eq:pyramidheight} 
  yields the desired estimate~\eqref{eq:pyramid}.
	
	Finally, the equality statement \eqref{eq:pyramidequality} follows
  from considering the equality cases in \eqref{eq:pyramidvolume}, 
  \eqref{eq:pyramidfunctionh} and~\eqref{eq:pyramidheight}: 
  $(L(G_q))^n/V(\Delta)$ is minimal if and only if $p_0-q\in P^\perp$, $z=z_0$,
  and $x_1=\ldots=x_k$, so that for all $i=1,\dots,k$
  \[
    h=\frac s{k^2} 
    = \frac1k\vert p_i-q\vert\quad\text{and}
    \quad z=\frac{s-nh}{n-1} 
    =\frac{k^2-n}{k(n-1)}\vert p_i-q\vert\,. 
  \]
  \nopagebreak
  \vspace{-10mm}
 
\end{proof}

\pagebreak[4]
\begin{proof}[Proof of Proposition~\ref{prop:nsimplex}]
	The left-hand side of \eqref{eq:nsimplexformula} is scaling invariant 
  so we may assume $V(\Delta)=1$.  Moreover,
  $L(G_0)$ is a continuous functions of $p_0,\ldots,p_n\in\R^n$, 
  and a minimizing sequence for $L(G_0)$ clearly has all $|p_i|$ bounded.
  Thus a minimizer~$\Delta$ for $(L(G_0))^n/V(\Delta)$ exists.
	
  We want to show that $\Delta$ is regular.
  For arbitrary $0\leq\ell\leq n$,
  regard the simplex~$\Delta$ as a pyramid with apex $p_\ell$
  and apply Lemma~\ref{lemma:pyramid} with $k=n$ and $q=0$.
  The first equations of~\eqref{eq:pyramidequality} give
	\begin{align*} 
		|p_0|=\dots=|p_n|\,,
	\end{align*}
	while the perpendicularity of $p_\ell$ to the hyperplane containing
  the other vertices gives
	\begin{align*}
		0=\langle p_\ell,p_i-p_j\rangle
    =\langle p_\ell,p_i\rangle-\langle p_\ell,p_j\rangle\quad
    \text{for all $i,j\neq\ell$}\,.
	\end{align*}
  We conclude the $n+1$ vertices are contained in a sphere and make pairwise
  equal angles when viewed from the origin. 
  Hence $\Delta$ is a regular simplex as stated.
	
	For a regular $n$-simplex $\Delta$, length and volume
  can be computed as the following functions of the edge length~$a$,
	\[
	  L(G_0)=a\,\frac{(n+1)}{\sqrt2}\sqrt{\frac n{n+1}}  \quad\text{and}\quad 
	  V(\Delta)=\frac{a^n}{n!}\sqrt{\frac{n+1}{2^n}}\,.
	\]
	Inserting these values into \eqref{eq:nsimplexformula} 
  gives the desired estimate.
\end{proof}
From the proposition we now derive an existence and uniqueness statement
which in particular applies to degree-$4$ networks in~$\R^3$
or to doubly periodic Steiner networks in~$\R^2$.
%
\begin{theorem}\label{thm:dipole}
	Let $N$ be an irreducible $n$-periodic network of degree $d=n+1$ 
  for $n\geq2$. Then its length quotient satisfies
	\begin{align}\label{eq:dipole}
		\frac{L^n}V\geq\sqrt{(n+1)^{n-1}n^n}\,.
	\end{align}
	Equality holds if and only if $N$ covers the dipole graph $D_{n+1}$ and 
  for each vertex $q\in N$ the leaves of $\Star q$ 
  form the vertices of a regular $n$-simplex.
\end{theorem}
\noindent
For $n=3$ this proves the standard diamond network with degree $d=4$ minimizes 
the length quotient, with $L^3/V=\sqrt{4^2\,3^3}=12\sqrt 3$. 
See Section~\ref{sec:degfour} for a complete discussion of the case $d=4$ 
in~$\R^3$.  Let us also note that for $n=2$ the Theorem confirms the
optimality of the hexagonal \hcb\ network, a fact we 
proved in \cite{periodicsteiner}.  
\begin{proof}
  By Proposition~\ref{prop:structure}, 
  the network $N$ covers the double bouquet graph $D_{\ell,k}$ 
  for some $2\leq k\leq n+1$ with $2\ell+k=n+1$. 
  Note that $D_{\ell,k}$ contains exactly
  \begin{align}\label{eq:dipolecycles}
	  2\ell+(k-1)=n = \rank N
  \end{align}
  cycles generating the first homology group; since $N$ is $n$-periodic
  they are independent, that is, each lifts to an independent 
  generator of~$\Lambda$.

  We remove from $N$ all edges projecting to the loops of $D_{\ell,k}$.
  From the remaining subset we consider a component~$N'\subset N$.
  The graph $N'$ covers~$D_k$ and so has degree~$k$. 
  Moreover, since each cycle of $D_{\ell,k}$ is independent,
  so is each of the $k-1$ generating cycles of~$D_k$. 
  Consequently $N'$ is a $(k-1)$-periodic network,
  contained in some $(k-1)$-dimensional affine subspace of~$\R^n$.

  Let $L'$ denote the length of~$N'$ 
  and $V'$ be its $(k-1)$-dimensional volume. 
  We claim
\begin{align}\label{eq:dipolelower}
  \frac{L^n}{V}\geq\frac{n^n}{(k-1)^{k-1}}\frac{(L')^{k-1}}{V'}
  \qquad\text{for}\quad k=2,\ldots,n+1.
\end{align}
  In case $k=n+1$ the quotient $N/\Lambda$ has no loops, 
  so that $N'=N$ and~\eqref{eq:dipolelower} is immediate. 
  Thus consider the case $k\leq n$.  
  Each of the $2\ell$ loops of $D_{\ell,k}$ gives rise to 
  a generator of~$\Lambda$, not contained in~$\Lambda'$.
  Moreover, the loops lift to straight edges $e_1,\ldots,e_{2\ell}$
  of~$N$ which are not contained in~$N'$. 
  These edges contribute length to~$N$, but not to~$N'$,
  and we can estimate
  \begin{align}\label{eq:dipoleloweredge}
		\frac{L^n}V\geq\frac{\big(L' + \vert e_1\vert + \cdots + \vert e_{2\ell}\vert\big)^n}{V'\,\vert e_1\vert\cdots\vert e_{2\ell}\vert}\,.
	\end{align}
	In terms of $x:=\sqrt[2\ell]{\vert e_1\vert\cdots\vert e_{2\ell}\vert}>0$ 
  the estimate on geometric and arithmetic mean yields
	\begin{align}\label{eq:dipolefunction}
		\frac{L^n}V\ge \frac{\big(L' + 2\ell x\big)^n}{V'x^{2\ell}}
    =\frac{\big(L'+(n+1-k)x\big)^n}{V'\,x^{n+1-k}}\,.
	\end{align}
	Regard the right-hand side of~\eqref{eq:dipolefunction} as a function 
  of $x\in(0,\infty)$, and differentiate to find the unique critical point at
	$x_0 = {L'}/({k-1})$.
	Moreover, the limit $x\to0$ verifies that $x_0$ is a minimum. 
  Insertion of $x_0$ into~\eqref{eq:dipolefunction} 
  proves our claim~\eqref{eq:dipolelower}.

We want to derive an explicit estimate from~\eqref{eq:dipolelower} 
which will show that $L^n/V$ can be estimated by its minimal value for $k=n+1$.
Pick a vertex $q\in N'$. 
Its $k$ neighbours $p_1,\ldots,p_k\in N'$ form the vertices of a $(k-1)$-simplex $\Delta$ (that is, a pyramid) with volume 
\[ 
  V(\Delta)=\frac{V'}{(k-1)!}\,.
\]
The length $L'$ of $N'$ coincides with the length of~$\Star q$.
We apply Proposition~\ref{prop:nsimplex} to $\Delta$ and conclude that the length quotient $L'^{k-1}/V'$ is minimal if and only if $\Delta$ is a regular $(k-1)$-simplex with $q$ the center of symmetry. 
  Estimating the right hand side of \eqref{eq:dipolelower} 
  with \eqref{eq:nsimplexformula} (for $n=k-1$) gives
	\begin{align}\label{eq:dipolek}
		\frac{L^n}{V}
    \geq n^n\sqrt{k^{k-2}(k-1)^{1-k}}
  \qquad\text{for}\quad k=2,\ldots,n+1.
	\end{align}
The right-hand side of~\eqref{eq:dipolek} is strictly decreasing in~$k$, 
and so $L^n/V$ can be estimated by the right hand side with $k=n+1$; 
in particular, \eqref{eq:dipole} holds.

Equality in \eqref{eq:dipolek} 
(and so in \eqref{eq:dipole}) can only hold for $k=n+1$,
in which case $N'=N$ and $N$ covers~$D_{n+1}$.
Our derivation shows that for $k=n+1$ equality in~\eqref{eq:dipole} 
holds precisely for the case that Proposition~\ref{prop:nsimplex}
holds with equality, namely for a regular $n$-simplex 
with $q$ the centre of symmetry.
\end{proof}
\begin{remark}
	For $2n>d>n+1$ an optimal $n$-periodic network of degree~$d$ does not necessarily cover~$D_d$. For example, in dimension $n=3$ the minimizer for $L^3/V$ among networks of degree $d=5$ is the \bnn\ network, 
covering $D_{1,3}$ (cf.~Tab.~\ref{tab:summary}).
\end{remark}
\pagebreak[4]

\section{Networks of degree \texorpdfstring{$d\geq2n$}{d>2n}}\label{sec:cube}

By Proposition~\ref{prop:structure}, an irreducible network 
of even degree~$d\geq2n$ covers the bouquet graph $B_{d/2}$. 
We estimate its length quotient:
\begin{theorem}\label{th:pcu}
	Let $N$ be an irreducible $n$-periodic network of even degree $d\geq2n$ 
  with lattice $\Lambda$. Then
	\begin{equation}\label{eq:dlargertwon}
    \frac{L^n}V\geq\Big(\frac d2-n+1\Big)n^n\,.
  \end{equation}
	Equality holds if and only if $d = 2n$ 
  and $\Lambda$ is similar to the primitive lattice~$\mathbb Z^n$.
\end{theorem}
\begin{figure}[b]
	\centering
	\includegraphics[width=.3\linewidth]{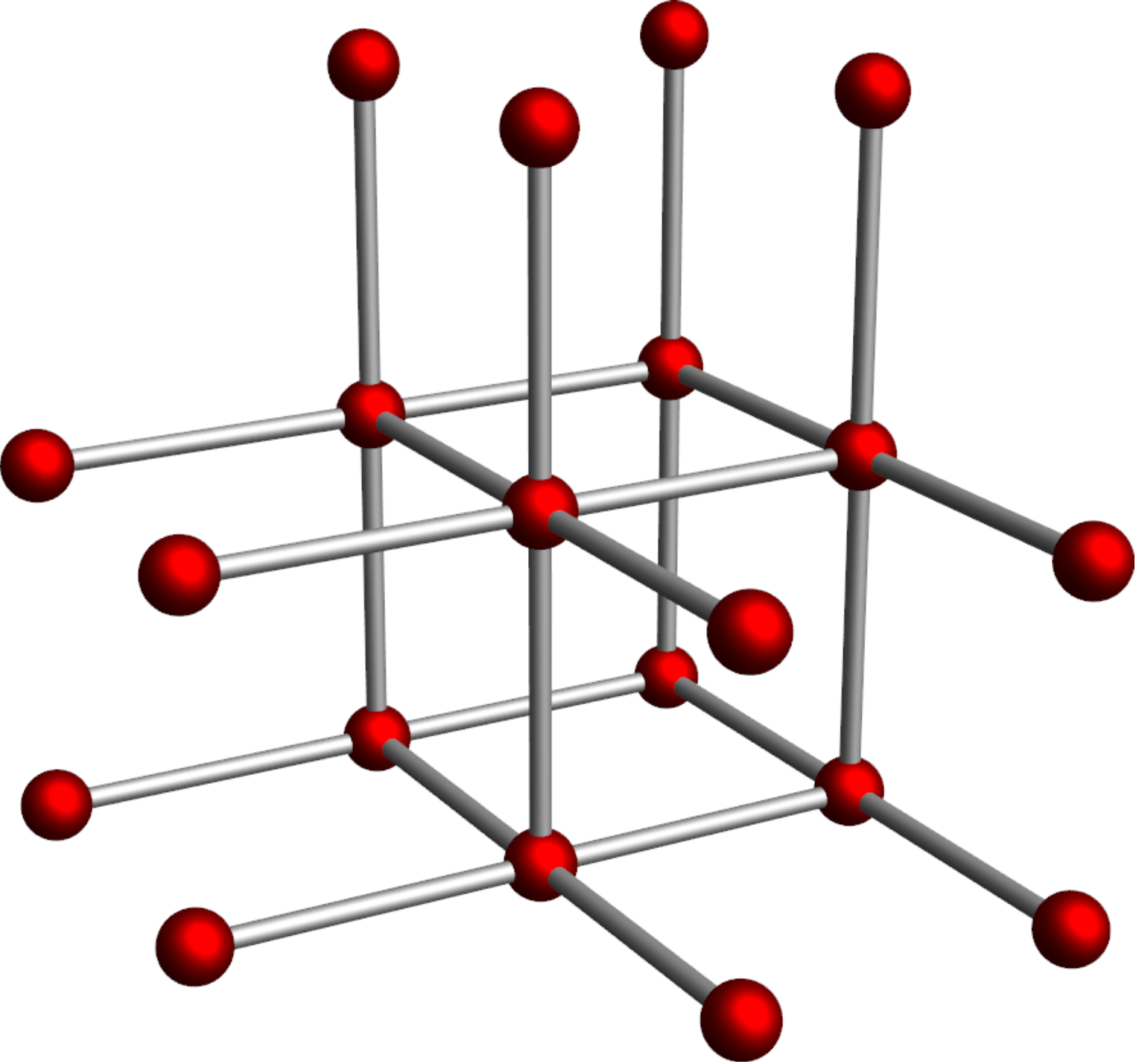}
	\caption{Among triply periodic networks of degree~$6$, 
    the \pcu\ network with quotient graph $B_3$ minimizes the length quotient. 
  } 
	\label{fig:network6}
\end{figure}
\noindent 
For $\R^3$ this settles the case~$d=6$: 
equality is attained by the \textsf{pcu} network
which has the edge set of a tessellation of 3-space with cubes 
(see Figure~\ref{fig:network6}).
Similarly, for $n=2$, the \textsf{sql} network relating to a square 
tessellation is optimal.
The estimate \eqref{eq:dlargertwon} implies that in each dimension~$n$ 
networks with even degree $d>2n$ have a length quotient larger than
for $d=2n$.
\begin{proof}
	Pick a vertex $p_0\in N$.
	We consider its neighbours $q_1,\ldots,q_d$ and set $g_i:=q_i-p_0$.
  Since $N$ covers the graph $B_{d/2}$
  we may assume the indexing is such that 
	the $n\leq d/2$ vectors $g_1,\ldots,g_n$ span $\Lambda$,  
  that $L=\sum_{i=1}^{d/2}\vert g_i\vert$,
	and that $\vert g_1\vert=\min_{1\leq i\leq d/2}\vert g_i\vert$. 
	The inequality on geometric and arithmetic mean gives
	\begin{align}
		\begin{split}
			V & = \vert\det(g_1,\ldots,g_n)\vert 
      \leq \frac1{\frac d2-n+1}\,\Big(\frac d2 - n+1\Big)\vert g_1\vert\cdot\vert g_2\vert\cdots\vert g_n\vert\\
			&\leq\frac1{\big(\frac d2-n+1\big)n^n}\Big(\big(\frac d2 - n+1\big)\vert g_1\vert + \vert g_2\vert +\ldots+\vert g_n\vert\Big)^n\,.
		\end{split}\label{eq:degree2n1}
	\end{align}
	Moreover, we use $n\leq d/2$ and $\vert g_1\vert\leq\vert g_i\vert$
  for $i=n+1,\ldots,d/2$ to obtain
	\begin{align}
		\begin{split}
			V\, \Bigl(\frac d2-n+1\Bigr)n^n 
      \leq\big(\vert g_1\vert+\vert g_2\vert+\ldots+\vert g_{d/2}\vert\big)^n
			=L^n\,.
		\end{split}\label{eq:degree2n2}
	\end{align}
	Let us show that equality cannot hold for $d\geq2n+2$.  If the second 
  inequality of~\eqref{eq:degree2n1} happens to be an equality, then
	\[
    \Big(\frac d2 - n + 1\Big)\vert g_1\vert 
    = \vert g_2\vert = \ldots = \vert g_n\vert\,.
  \]
	In particular, $\vert g_2\vert,\ldots,\vert g_{d/2}\vert$ are strictly larger
	than $\vert g_1\vert$ and equality cannot hold in~\eqref{eq:degree2n2}.
	For $d=2n$, however, 
  equality holds if and only if $g_1,\ldots,g_n$ 
  are pairwise perpendicular and have the same length, 
  i.e., $\Lambda$ is the primitive $n$-dimensional lattice.
\end{proof}
\begin{remark}\label{rem:evenmonotone}
	The construction of the proof of Theorem~\ref{th:pcu} shows 
  the length quotient of irreducible $n$-periodic networks is strictly increasing
  when restricted to even degree~$d\geq2n$:
  Removal of an edge of the quotient network~$N/\Lambda$ and thereby degree reduction by~$2$ 
  decreases length while not affecting balancing.
\end{remark}
The theorem leaves open the case of networks with odd degree. 
We present an estimate for the length quotient for that case, 
which is weaker than~\eqref{eq:dlargertwon}:
\begin{theorem}\label{thm:highdeg}
	If $N$ be an irreducible $n$-periodic network of degree $d\geq n+1$ then
		\begin{align}\label{eq:highdeg}
			\frac{L^n}V \ge \sqrt{(n+1)^{n-1}\,n^n}\,.
		\end{align}
	Equality holds if and only if $d=n+1$ and $N$ covers the dipole graph $D_{n+1}$ 
  and for each vertex $q\in N$ the leaves of $\Star q$ 
  form the vertices of a regular $n$-simplex.
\end{theorem}
\noindent
For $n=3$ we will obtain a stronger estimate in Corollary~\ref{cor:highdegthree}.
\begin{proof}
	For $n\geq2$ and even degree $d\geq2n$, Theorem~\ref{th:pcu} gives
	\[\frac{L^n}V\geq\Big(\frac d2-n+1\Big)n^n\geq n^n\,.\]
	On the other hand, invoking $(1+1/n)^n < 3$ and $n\ge 2$ gives
  \[
    n^n > n^n \sqrt{\frac13\Big(1+\frac1n\Big)^n}
    \ge\sqrt{\frac1{n+1}\,{(n+1)^n}\,n^n}\,,
  \]
	which implies~\eqref{eq:highdeg} with strict inequality. 

  For all other $d\geq n+1$, Proposition~\eqref{prop:structure} identifies 
  the topology of $N/\Lambda$ as a double bouquet graph $D_{\ell,k}$,
  where $k\geq2$.
	Pick an $n$-periodic subnetwork $N'\subset N$ (with the same lattice) 
  subject to the following property: 
  The removal of any edge from $N'/\Lambda$ disconnects the covering network~$N'$. 
  Then $N'/\Lambda$ decomposes into two bouquet graphs 
  $B_{\ell_1}$ and $B_{\ell_2}$, connected with $1\leq k'\leq k$ edges. 
  This graph contains exactly $\ell_1+\ell_2+(k'-1)=n$ cycles. 

	In case $k'=1$ of one bridge the network $N'$ contains $\ell_1+\ell_2=n$ loops. 
  Keeping the lattice, 
  we can decrease the length of the bridge to~$0$, to obtain from~$N'$
  a network~$N''$ of smaller length, covering the bouquet graph~$B_n$ of degree~$2n$. 
  Applying Theorem~\eqref{th:pcu} yields~\eqref{eq:highdeg} for~$N''$.
  Due to $L(N)\geq L(N')> L(N'')$ the estimate~\eqref{eq:highdeg} 
  with strict inequality follows for~$N$.
 
  In the other case $k'\geq2$, let us first assume $d=n+1$. 
  Since~$N$ has exactly two vertices, 
  by~\eqref{eq:numbervertices} the removal of any edge disconnects the network. So~$N'=N$.
  Then we can apply the reasoning of the proof of Theorem~\ref{thm:dipole} to~$N$,   
  replacing $2\ell$ by $\ell_1+\ell_2$ and taking $k'$ for~$k$. 
  This yields the estimate~\ref{eq:highdeg} for $N$ 
  and characterizes the equality case as claimed.
  If, on the other hand, $d>n+1$
  the subnetwork $N'$ is obtained by removing at least one edge from~$N$.   
  Thus $L(N)>L(N')$.
  Moreover, as the estimate~\eqref{eq:highdeg} 
  is established for $N'$ with $d=n+1$, 
  it follows for $N$ with strict inequality.
\end{proof}
\pagebreak[4]

\section{Triply periodic networks of degree 4}\label{sec:degfour}

\begin{figure}[b]
	\centering
	\begin{subfigure}[t]{.52\linewidth}
		\centering
		\includegraphics[width=\linewidth]{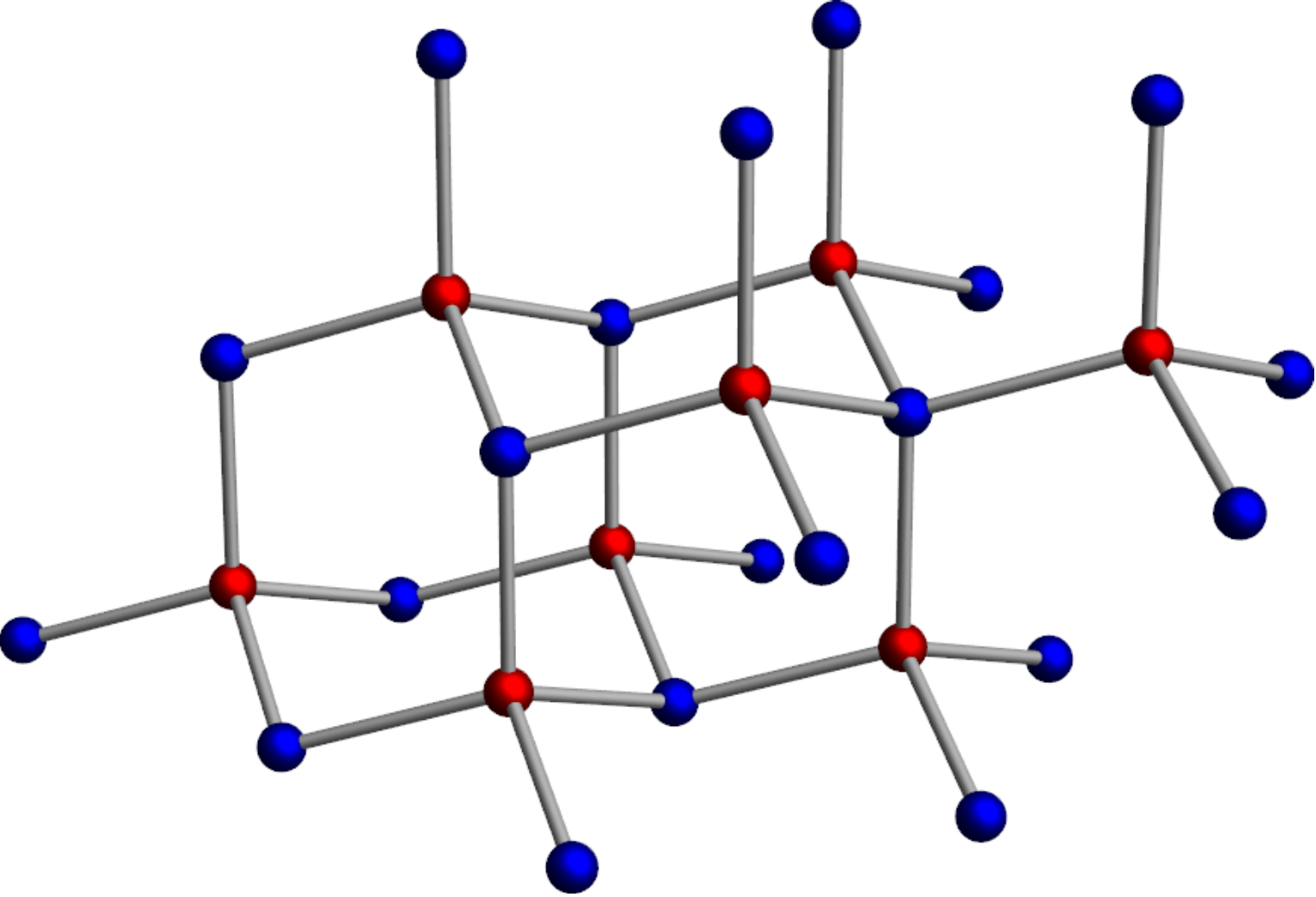}
	\end{subfigure}\hfill
	\begin{subfigure}[t]{.46\linewidth}
		\centering
		\includegraphics[width=.57\linewidth]{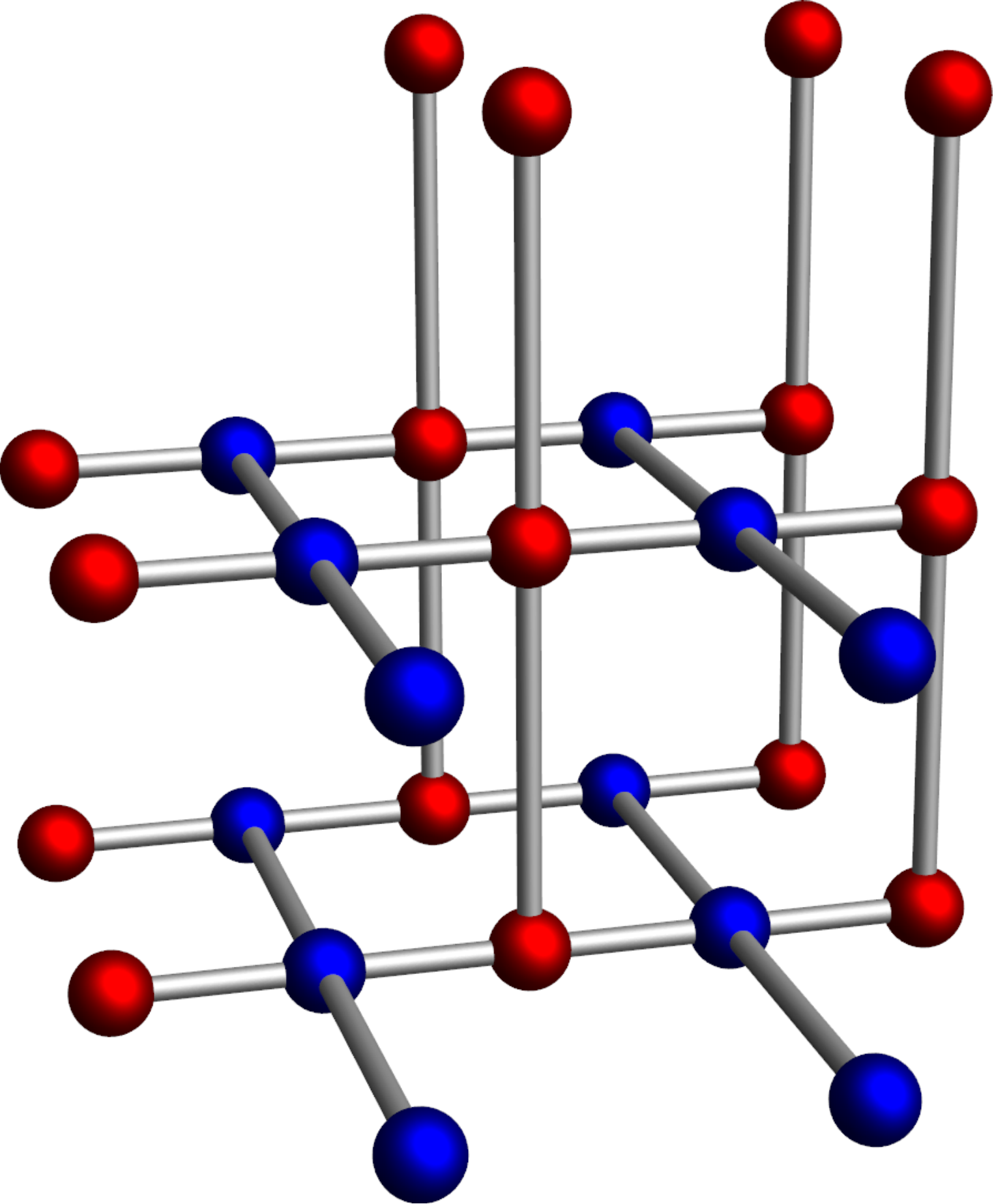}
	\end{subfigure}
	\caption{Among irreducible triply periodic networks of degree~$4$, the diamond network shown left minimizes the length quotient; it covers the dipole graph~$D_4$. 
The other graph with degree~4 on two vertices 
is $D_{1,2}$;
a minimizing \cds\ network is shown on the right.}
	\label{fig:network4}
\end{figure} 

In the remainder of the paper, 
we study specifically the case of three dimensions.
By Proposition~\ref{prop:structure}, 
an irreducible triply periodic network of degree $d=4$ 
must have a quotient with two vertices and four edges 
which is either the dipole graph~$D_4$ 
or the double bouquet graph~$D_{1,2}$. 
Theorem~\ref{thm:dipole} asserts the absolute minimizer 
for the length quotient $L^3/V$ with degree~$4$ 
covers~$D_4$ and is the \emph{diamond} network \dia; 
it is uniquely determined up to similarities of~$\R^3$. 
This is included as part \ref{item:dia} of the following statement,
while part \ref{item:cds} determines the optimal embedding covering~$D_{1,2}$.
See Figure~\ref{fig:network4} as well as Figures~\ref{fig:dia} and \ref{fig:cds}.

\begin{theorem}\label{thm:diacds}
  Let $N\subset\R^3$ be an irreducible triply periodic network of degree~$4$.
  \begin{enumerate}
		\item\label{item:dia} If $N$ covers $D_4$, then
		\begin{align}\label{eq:dia}
			\frac{L^3}V\geq12\sqrt3\approx20.8\,.
		\end{align}
		Equality holds if and only if all edge lengths of $N$ are equal 
		and the lattice~$\Lambda$ is face-centered cubic, i.e.,
    precisely for the diamond network \dia.
		\item\label{item:cds} If $N$ covers $D_{1,2}$, then
		\begin{align}
			\frac{L^3}{V}\geq27\,.
			\label{eq:cds}
		\end{align}
		Up to similiarity, equality is attained by a $1$-parameter family 
    of networks with primitive lattice $\Lambda$; 
    we label these networks \cds.
	\end{enumerate}
\end{theorem}
\noindent
In Figure~\ref{fig:network4} the blue vertices are placed exactly 
in the middle between the red vertices.
The one-parameter family of \cds-networks corresponds to 
translating the set of blue vertices on their edges to the red vertices.
Clearly, this leaves the length invariant.
We note that the depicted network coincides with the unique minimizer 
of the energy~$\big(\sum x_i^2\big)^{3/2}/V$ studied by Sunada
(see~\cite{sunada2012topological}).

	\begin{figure}
		\centering
		\begin{subfigure}[c]{.42\linewidth}
			\centering
			\begin{tikzpicture}[x=3cm,y=3cm]
				\coordinate[dot1](p) at (0,0);
				\coordinate[dot2](q) at (1,0);
				\draw[steiner1](p) to [out=80,in=100] (q);
				\draw[steiner2](p) to [out=30,in=150] (q);
				\draw[steiner3](p) to [out=-30,in=210] (q);
				\draw[steiner4](p) to [out=-80,in=260] (q);
			\end{tikzpicture}
		\end{subfigure}
		\begin{subfigure}[c]{.38\linewidth}
			\centering
			\begin{tikzpicture}[x=.8cm,y=.8cm]
				\coordinate[dot1](p0) at (0,0);
				\coordinate[dot2](q1) at (-2,0);
				\coordinate[dot2](q2) at (1,2);
				\coordinate[dot2](q3) at (2,-1);
				\coordinate[dot2](q4) at (-1,-2);
				\draw[steiner1](p0) -- node[below,pos=.55]{$x_1$} (q1);
				\draw[steiner2](p0) -- node[right]{$x_2$} (q2);
				\draw[steiner3](p0) -- node[above,pos=.6]{$x_3$} (q3);
				\draw[steiner4](p0) -- node[right,pos=.6]{$x_4$} (q4);
				\node[left ] at (q1) {$q_1$};
				\node[above] at (q2) {$q_2$};
				\node[right] at (q3) {$q_3$};
				\node[below] at (q4) {$q_4$};
				\node[above left] at (p0) {$p_0$};
			\end{tikzpicture}
		\end{subfigure}
		\caption{Topology and embedding of the dipole graph $D_4$.}
		\label{fig:dia}
	\end{figure}
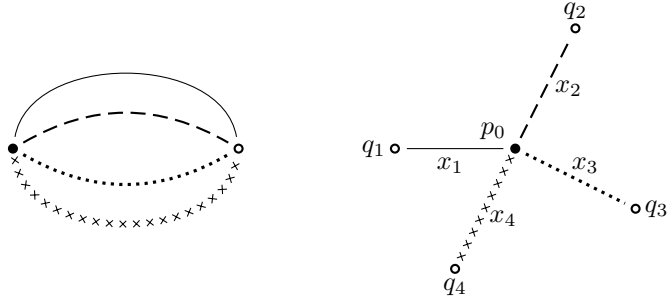
\begin{proof}
	It remains to prove \ref{item:cds}.  
	We take a subgraph of $N$ consisting of two adjacent vertices $p_1,q_1$ 
	and their neighbours $p_2,p_3,q_2$ so that the vertices $p_1,p_2,p_3$ 
	and the vertices $q_1,q_2$ are identified in the lattice~$\Lambda$,
	see Figure~\ref{fig:cds}.  
	\begin{figure}[b]
		\centering
		\begin{subfigure}[t]{.55\linewidth}
			\centering
			\begin{tikzpicture}[scale=2.5]
				\coordinate[dot2](p) at (0,0);
				\node[below,yshift=-1ex] at (p) {$p$};
				\coordinate[dot1](q) at (.75,0);
				\node[below,yshift=-1ex] at (q) {$q$};
				\draw[steiner1](p) to [out=45,in=135] (q);
				\draw[steiner2](p) to [out=-45,in=225] (q);
				\draw[steiner3](p) to [out=135,in=225,looseness=25] (p);
				\draw[steiner4](q) to [out=45,in=-45,looseness=25] (q);
			\end{tikzpicture}
		\end{subfigure}
		\hfill
		\begin{subfigure}[t]{.43\linewidth}
			\centering
			\begin{tikzpicture}[scale=1.5]
				\coordinate[dot2](p1) at (0,0);
				\coordinate[dot1](q1) at (0,1);
				\coordinate[dot2](p2) at (1,0);
				\coordinate[dot1](q2) at (1,1.5);
				\coordinate[dot2](p3) at (0,1.7);
				\node[below right] at (q1) {$q_1$};
				\node[right] at (q2) {$q_2$};
				\node[below right] at (p1) {$p_1$};
				\node[right] at (p2) {$p_2$};
				\node[above] at (p3) {$p_3$};
				\draw[steiner2]($(p1) - .5*(q1)$) -- (p1);
				\draw[steiner1](p1) -- node[left] {$x_4$} (q1);
				\draw[steiner3]($(p1) - .6*(p2)$) -- (p1) -- node[above] {$x_1$} (p2);
				\draw[steiner4]($1.6*(q1) - .6*(q2)$) -- (q1) -- node[above,yshift=0.5ex] {$x_2$} (q2);
				
				\draw[steiner2](q1) -- node[left]{$x_3$} (p3);
			\end{tikzpicture}
		\end{subfigure}
		\caption{Topology and embedding of the double bouquet graph $D_{1,2}$.}
		\label{fig:cds}
	\end{figure}
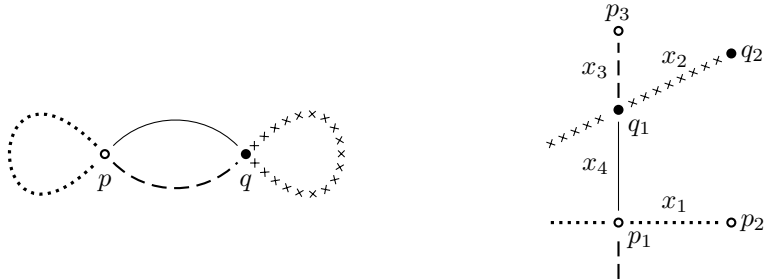
	The lattice $\Lambda$ is generated by the lift of three loops of $D_{1,2}$, 
  and so can be generated by 
	\begin{align*}
		g_1 &:= p_2 - p_1\,, & g_2 &:= q_2 - q_1\,, & g_3 &:= p_3 - p_1.
	\end{align*}
	The enclosed volume of $N$ can be estimated by
	\begin{align*}
		V(\R^3/\Lambda) &= \vert\det(g_1,g_2,g_3)\vert\\
		&\leq\vert g_1\vert\cdot\vert g_2\vert\cdot\vert g_3\vert\\
		&\leq\frac{1}{27}\big(\vert g_1\vert+\vert g_2\vert+\vert g_3\vert\big)^3\\
		&\leq\frac{1}{27}\big(\vert p_2 - p_1\vert + \vert q_2 - q_1\vert + \vert p_3 - q_1\vert + \vert q_1 - p_1\vert\big)^3\\
		&=\frac{1}{27}L^3(N/\Lambda)\,.
	\end{align*}
	Equality holds if and only if the $g_i$'s are pairwise perpendicular, 
  have the same length 
  and $q_1$ lies on the straight segment between $p_1$ and~$p_3$. 
	This implies the lattice is primitive and
  the edge lengths~$x_i$ given as in Figure~\ref{fig:cds} 
	satisfy $x_1=x_2=x_3+x_4>0$.
	In particular, equality for a fixed volume constraint $V=1$ 
  is attained by a 1-parameter family, parameterized by $x_4\in(0,x_2)$, say.
\end{proof}
	A \cds\ network with $x_3=x_4$ is shown in Figure~\ref{fig:network4}.
	In the two limits $x_3\to0$ and $x_4\to0$ the \cds\ network 
	degenerates to the \pcu\ network of degree~$6$.

\section{Triply periodic networks of degree 5}\label{sec:degfive}

Determining an optimal network of degree $5$ is more difficult 
than the case of degree~$4$.  This is due to the fact that 
an irreducible quotient graph~$\Gamma$ has $5$~edges and so
its fundamental group is generated by $4$~elements.  
Thus, one of the generators for $N$ 
must be contained in the lattice generated by the other three. 
This presents an integer constraint for our length optimization problem.
\begin{figure}[b]
	\centering
	\begin{subfigure}{.49\linewidth}
		\centering
		\includegraphics[width=\linewidth]{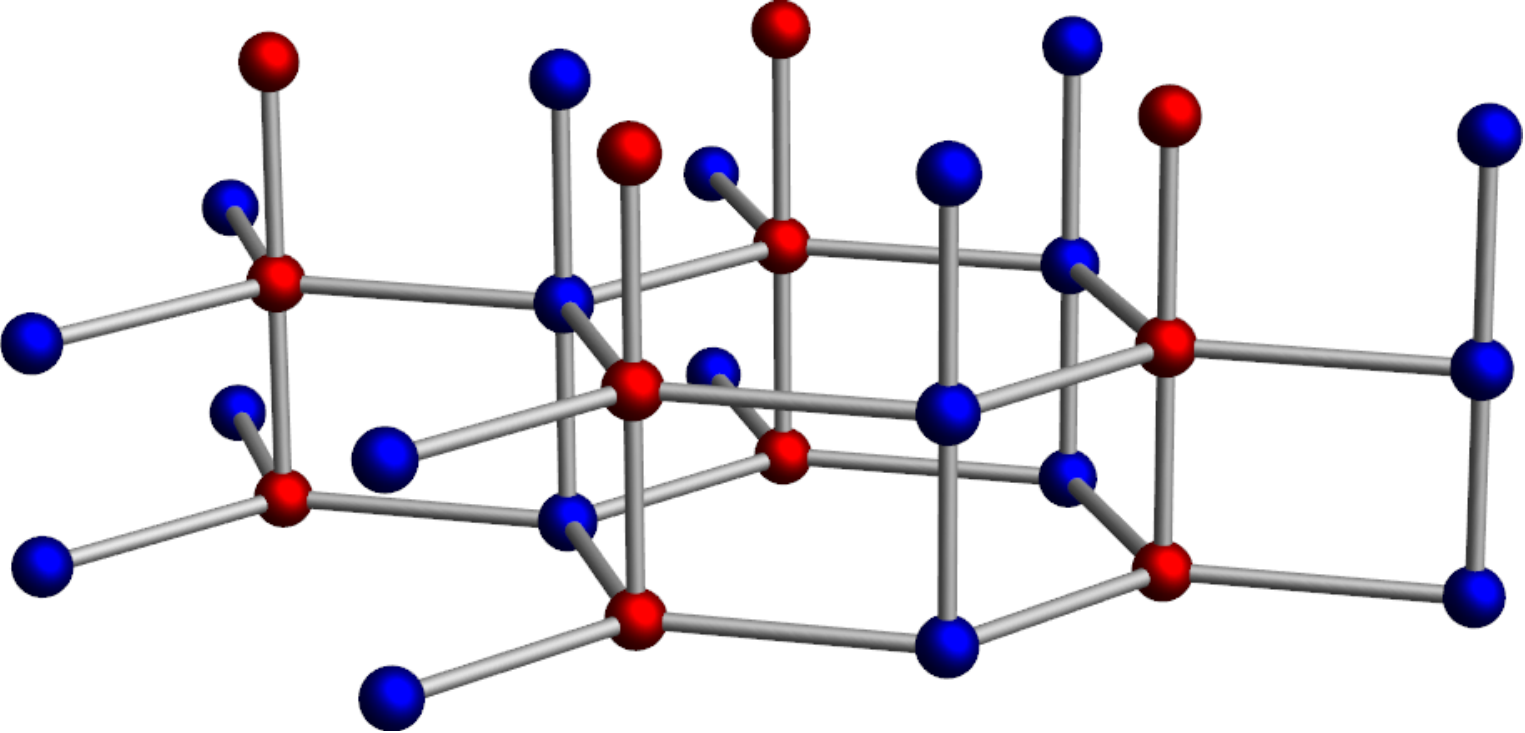}
	\end{subfigure}\hfill
	\begin{subfigure}{.49\linewidth}
		\centering
		\includegraphics[width=.7\linewidth]{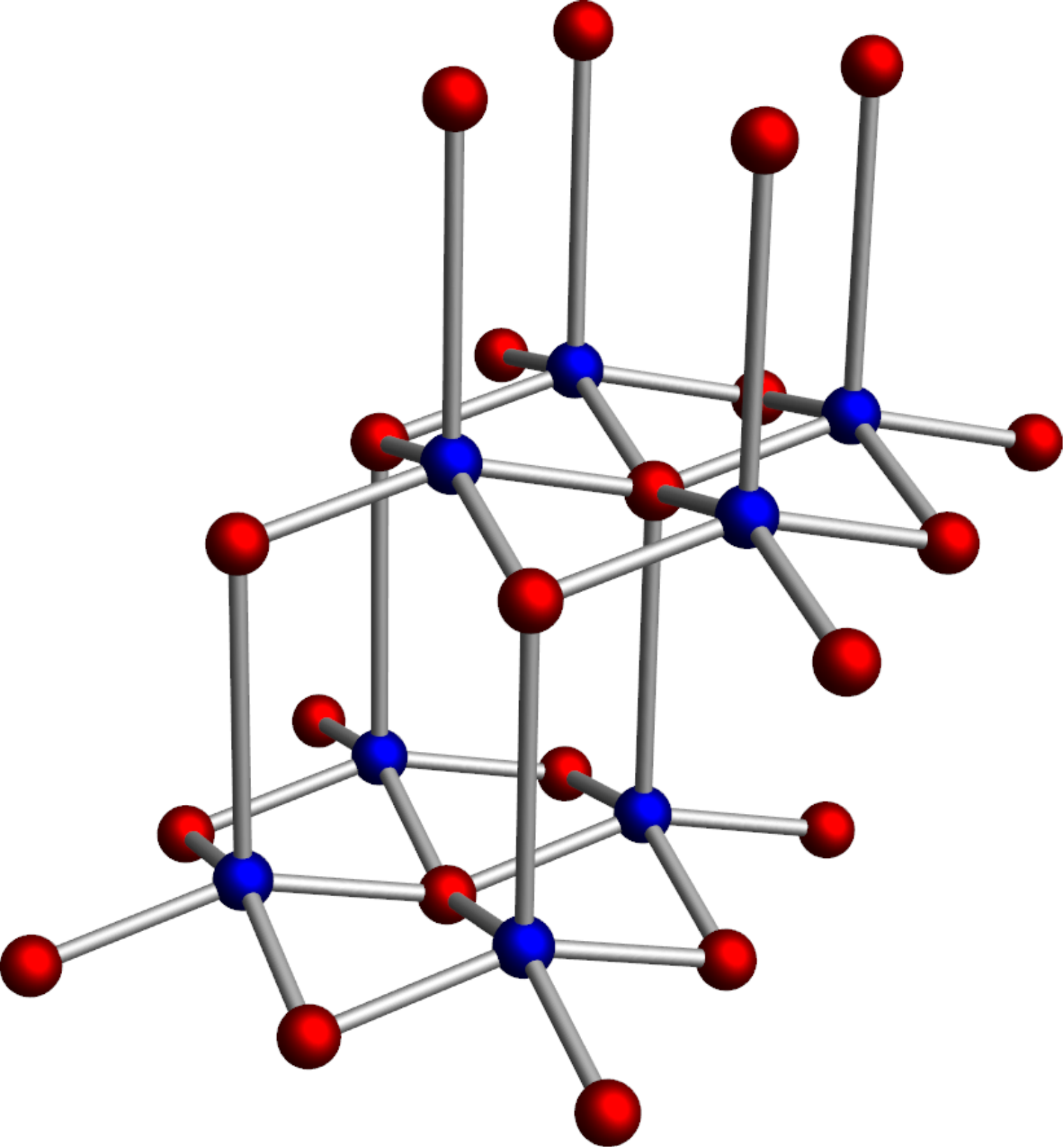}
	\end{subfigure}
	\caption{Among irreducible triply periodic networks of degree~$5$, 
  the \bnn\ network with quotient $D_{1,3}$ minimizes the length quotient (left). 
The other possible quotient is the dipole graph $D_5$, 
for which the \sqp\ network minimizes (right).}
	\label{fig:network5}
\end{figure}

According to Proposition~\ref{prop:structure}, 
an irreducible network of degree~$5$ 
can only attain the topologies $D_5$ or~$D_{1,3}$
depicted in Figure~\ref{fig:bnn} and \ref{fig:sqp}.
The network with smallest length quotient turns out to be 
a network covering~$D_{1,3}$ which we call \bnn. 
It corresponds to the edges of a tessellation of~$\R^3$ with hexagonal prisms, 
i.e., it contains parallel layers 
of minimizing doubly periodic hexagonal networks, see Figure~\ref{fig:network5}. 
\begin{theorem}\label{thm:bnn}
	If $N$ is an irreducible triply periodic network of degree $5$
  covering $D_{1,3}$ then
	\begin{align}\label{eq:bnn}
		\frac{L^3}{V}\geq27\sqrt3\approx46.8\,.
	\end{align}
	In the equality case, $N$ is the \bnn\ network with a hexagonal lattice:
  the network consists of prismatic honeycombs over regular hexagons,
  where the prism height equals $3/4$ of the hexagon edge length.
\end{theorem}
\begin{proof}
	Consider two vertices, labelled $p_0,q_1\in N$,
	which project to the two distinct vertices $p,q$ of~$D_{1,3}$. 
	Consider first the neighbours of the point~$p_0$, see Figure~\ref{fig:bnn}. 
  The loop endpoints in $D_{1,3}$ correspond to 
  two neighbours $p_1,p_2$ of~$p_0$, which project again to~$p$.  
  The three edges of $D_{1,3}$ 
  give rise to three further neighbours $q_1,q_2,q_3$, projecting onto~$q$. 
	The edges from $p_0$ to $p_1$ and~$p_2$ are opposite at~$p_0$ 
  and contained in a line~$\ell$.
\begin{figure}
	\centering
	\begin{subfigure}[t]{.55\linewidth}
		\centering
		\begin{tikzpicture}[x=3cm,y=3cm]
			\coordinate[dot1](p) at (0,0);
			\coordinate[dot2](q) at (.75,0);
				\node[below,yshift=-1ex] at (p) {$p$};
				\node[below,yshift=-1ex] at (q) {$q$};
			\draw[steiner1](p) to [out=45,in=135] (q);
			\draw[steiner2](p) to [out=-45,in=225] (q);
			\draw[steiner3](p) to [out=135,in=225,looseness=30] (p);
			\draw[steiner4](q) to [out=45,in=-45,looseness=30] (q);
			\draw[steiner5](p) to (q);
		\end{tikzpicture}
	\end{subfigure}
	\hfill
	\begin{subfigure}[t]{.43\linewidth}
		\centering
		\begin{tikzpicture}[x=1.5cm,y=1.5cm]
			\coordinate[dot1](p0) at (0,0);
			\coordinate[dot1](p1) at (0,1.2);
			\coordinate[dot1](p2) at (0,-1.2);
			\coordinate[dot2](q1) at (.5,.866);
			\coordinate[dot2](q2) at (.5,-.866);
			\coordinate[dot2](q3) at (-1,0);
			\coordinate[dot2](q4) at (.5,1.866);
			\draw[steiner1](p0) -- node[right] {$x_1$} (q1);
			\draw[steiner2](p0) -- node[right] {$x_2$} (q2);
			\draw[steiner5](p0) -- node[above] {$x_3$} (q3);
			\draw[steiner3](p0) -- node[left] {$y$} (p1);
			\draw[steiner3](p0) --  (p2);
			\draw[steiner4](q1) -- node[right] {$z$} (q4);
			\node[right] at (p0) {$p_0$};
			\node[right] at (q1) {$q_1$};
			\node[right] at (q2) {$q_2$};
			\node[left] at (q3) {$q_3$};
			\node[right] at (q4) {$q_4$};
			\node[left] at (p1) {$p_1$};
			\node[left] at (p2) {$p_2$};
		\end{tikzpicture}
	\end{subfigure}
	\caption{Topology and embedding of the double bouquet graph $D_{1,3}$.}
	\label{fig:bnn}
\end{figure}
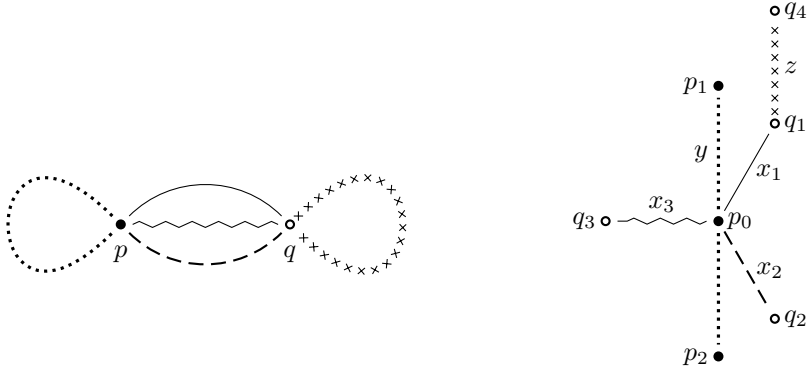
  
  We claim that it is sufficient to verify the theorem for $N$ balanced.  
  Note first that for a network with two vertices in the quotient,
  balancing at one vertex is equivalent to balancing at the other vertex.
  Suppose now $N$ is not balanced.  Then $N$ is not balanced at $p_0$,
  and so replacing $p_0$ with the Fermat point $F$ 
  of the triangle $q_1,q_2,q_3$ yields a balanced network with strictly
  smaller length, but with the same lattice and volume.  
  Possibly, the resulting network is not immersed, namely in case $q_1,q_2,q_3$
  are collinear, or the triangle $q_1,q_2,q_3$ 
  has an interior angle of at least $120$ degrees.
  In that case, however, $F$ coincides with one of the vertices $q_1,q_2,q_3$, 
  and so $N$ can be regarded as a network covering the bouquet graph $B_4$. 
  Applying Theorem~\eqref{th:pcu} gives $L^3/V\geq(4-3+1)\,3^3=54$, 
  so that \eqref{eq:bnn} holds strictly.

  Balancing at $p_0$ implies that $q_1,q_2,q_3$ must be coplanar with~$p_0$,
  thereby defining a plane~$P$.
  The same reasoning applies to the three neighbours of~$q_1$ projecting to~$p$,
  they define a plane~$P'$.
  The edge triples defining $P$ and $P'$ agree up to the translation
  from $p_0$ to $q_1$, and so $P=P'$.

  Consider now the line~$\ell'$ through~$q_1$
  determined by its two neighbours projecting to~$p$.
  For the lattice to have rank~$3$, 
  at least one of the lines $\ell,\ell'$ must be transverse to the plane~$P$. 
  Hence $p_1-p_0$ or $q_4-q_1$ is a generator of the lattice. 
  By relabelling let us assume $p_1-p_0$ has this property. 

  The points $q_1,q_2,q_3$ are not collinear 
  and define a triangle with positive area $A_\Delta$.
  Thus the volume $V$ of $N/\Lambda$ satisfies 
 	\begin{align}\label{eq:bnnvolume}
  		V \leq 2A_\Delta\dist(p_1,P)\,,
 	\end{align}
 	Equality in~\eqref{eq:bnnvolume} is attained if and only if 
  from the four generators of the homology of $D_{1,3}$,
	\[g_1 := q_1 - q_3\,,\quad g_2 := q_2 - q_3\,, \quad 
    g_3 := p_1 - p_0\,,\quad g_4 := q_4-q_1\,,\]
  the first three span the lattice $\Lambda$.

	Setting $x_i := \vert q_i - p_0\vert$ for $i=1,2,3$, 
	and $y := \vert p_1 - p_0\vert$, $z := \vert q_4 - q_1\vert$ 
	we have $L=x_1+x_2+x_3+y+z$. 
	We may assume a choice of coordinates with $p_0 = 0$ and
	\begin{align}\label{eq:bnncoordinates}
		q_1 = x_1\begin{pmatrix}-1\\0\\0\end{pmatrix}\,,\quad 
		q_2 = \frac{x_2}2\begin{pmatrix}1\\{\sqrt3}\\0\end{pmatrix}\,,\quad 
		q_3 = \frac{x_3}2\begin{pmatrix}1\\-{\sqrt3}\\0\end{pmatrix},
	\end{align}
	which gives
	\begin{align}\label{eq:bnnarea}
		2A_\Delta
		=\bigl\vert\det(q_1-q_3,\,q_2-q_3)\bigr\vert
    =\frac{\sqrt3}2\bigl(x_1x_2+x_1x_3+x_2x_3\bigr)\,.
	\end{align}
	We now distinguish the case $q_4\in P$ from $q_4\notin P$.
		\begin{figure}
			\centering
			\begin{tikzpicture}
				\coordinate[dot1](p0) at (0,0);
				\coordinate[dot2](q1) at (60:1);
				\coordinate[dot2](q2) at (300:1.5);
				\coordinate[dot2](q3) at (180:2);
				\coordinate[dot1](q7) at ($(q1)-(q3)$);
				\coordinate[dot1](q5) at ($(q1)-(q2)$);
				\coordinate[dot1](q6) at ($(q2)-(q3)$);
				\coordinate[dot2](q4) at ($(q7)+(q2)$);
				\draw[steiner2](p0) -- node[right]{$x_2$} (q2);
				\draw[steiner1](p0) -- node[left]{$x_1$} (q1);
				\draw[steiner5](p0) -- node[above]{$x_3$} (q3);
				\draw[steiner2](q1) -- (q5);
				\draw[steiner5](q1) -- (q7);
				\draw[steiner5](q5) -- ($(q5)+.5*(q3)-.5*(p0)$);
				\draw[steiner2](q4) -- (q7);
				\draw[steiner5](q2) -- (q6);
				\draw[steiner2](q3) -- ($(q3)-.6*(q2)$);
				\draw[steiner1](q3) -- ($(q3)-.5*(q1)$);
				\draw[steiner1](q2) -- ($(q2)-.6*(q1)$);
				\draw[steiner1](q6) -- (q4);
				\draw[steiner4]($1.4*(q2) -.4*(q3)$) -- (q2) -- (q3) -- ($1.4*(q3)-.4*(q2)$);
				\draw[steiner4](q4) -- node[below]{$z$} (q1) 
					--+ ($.8*(q1)-.8*(q4)$);
				\node[right] at (p0) {$p_0$};
				\node[above right] at (q3) {$q_3$};
				\node[above right] at (q1) {$q_1$};
				\node[above right] at (q2) {$q_2$};
				\node[right] at (q4) {$q_4$};
			\end{tikzpicture}
			\caption{$P\cap N$ in \ref{item:bnnhorizontal}, where $q_4$ is contained in the plane $P$ spanned by $q_1,q_2,q_3$.}
			\label{fig:bnnhorizontal}
		\end{figure}
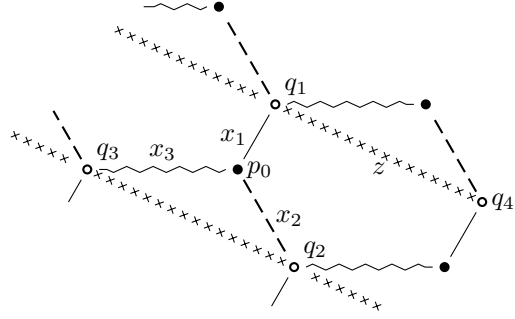
	\begin{caselist}
		\item\label{item:bnnhorizontal} Suppose $q_4\in P$ 
			(cf.~Fig~\ref{fig:bnnhorizontal}). 
      In $\R^3/\Lambda$ the vertex $q_4$ and $q_1,q_2,q_3$ are identified,
			and so in $\R^3$ the smallest lattice vector contained in $P$ gives a 
      lower bound for $\vert q_4-q_1\vert$.  
      For our hexagonal lattice $\Lambda\cap P$ this gives
		\[\vert q_4-q_1\vert\geq\min\big\lbrace\vert q_2-q_1\vert,\vert q_3-q_1\vert,\vert q_3 - q_2\vert\big\rbrace\,.\]
			By relabelling we may assume 
			$\vert q_4 - q_1\vert \geq \vert q_2 - q_1\vert$. 
      This inequality and
      the geometric arithmetic mean inequality give
			\begin{align*}
        z \geq \vert q_2 - q_1\vert 
				= \sqrt{\bigl(x_1+\frac12x_2\bigr)^2+\frac34x_2^2} 
				= \sqrt{\bigl(x_1+x_2\bigr)^2-x_1x_2}
         \ge \frac{\sqrt3}2 \bigl(x_1+x_2\bigr).
			\end{align*}
			Thus we can estimate $s := x_1+x_2+x_3+z$ as
			\begin{align*}
				s \geq 
				\frac{2+\sqrt3}2(x_1+x_2)+x_3\,.
			\end{align*}
		  Moreover, estimating $x_1x_2$ in \eqref{eq:bnnarea} gives
			\begin{align*}
				2A_\Delta 
				\leq\frac{\sqrt3}2\Big(\frac14(x_1+x_2)^2+(x_1+x_2)x_3\Big)\,.
			\end{align*}
			We combine the last two inequalities to arrive at
			\begin{align}
				\frac{s^2}{2A_\Delta} \geq \frac{\big((2+\sqrt3)(x_1+x_2)+2x_3\big)^2}{2\sqrt3\big(\frac14(x_1+x_2)^2+(x_1+x_2)x_3\big)}\,.\label{eq:bnnhorizontal}
			\end{align}
			Let us determine the minimum of the right-hand side 
			of~\eqref{eq:bnnhorizontal}. 
			Using scaling invariance of this quotient and $x_1+x_2>0$ 
			we may assume $x_1+x_2=1$.  So it suffices to minimize
			\[
        x_3\mapsto\frac1{2\sqrt3}\frac{\big(2+\sqrt3+2x_3\big)^2}{\frac14+x_3}
        \qquad\text{for }x_3>0\,.
      \]
      This function attains its minimal value $2(2+\sqrt3)$ 
			at $x_3=(1+\sqrt3)/2$, and so
			\begin{align*}
				\frac{s^2}{2A_\Delta}\geq2\,(2+\sqrt3)\,.
			\end{align*}
      Inserting this estimate into~\eqref{eq:bnnvolume} and then 
      using an estimate on the geometric mean of the kind 
			$a\big(\frac b2\big)^2\leq\big(\frac{a+b}3\big)^3$ 
      verifies \eqref{eq:bnn} strictly
      (so that equality cannot be attained):
			\begin{align*}
				V &\leq 2A_\Delta\dist(p_1,P)
				\leq\frac1{2(2+\sqrt3)}(x_1+x_2+x_3+z)^2\,y
				\leq\frac2{27(2+\sqrt3)}L^3\,.
			\end{align*}
		\item\label{item:bnnvertical} 
			Suppose $q_4\notin P$ so that $q_4$ lies in $\Lambda\setminus P$. 
      Since $g_1,g_2,g_3$ generate the lattice
			the edge length $z$ is at least $\dist(P,p_1)$, and so
		\begin{align}\label{eq:bnndistance}
		 	2\dist(p_1,P)\leq y+z\,.
		\end{align}
		On the other hand, estimating~\eqref{eq:bnnarea} by the
    arithmetic and geometric mean inequality gives
		\begin{align}\label{eq:bnnvertical}
      4\sqrt3\,A_\Delta
      =3(x_1x_2+x_1x_3+x_2x_3)
			\leq(x_1+x_2+x_3)^2 \,.
		\end{align}
		Then the inequality resulting from \eqref{eq:bnnvolume} 
    and~\eqref{eq:bnnarea}
    can be estimated first using \eqref{eq:bnndistance} 
    and~\eqref{eq:bnnvertical}. 
		Finally, the estimate on the geometric mean of the kind used before 
    yields the desired inequality~\eqref{eq:bnn}:
		\begin{align}\label{eq:bnnverticalfinal}
			V \leq \frac1{4\sqrt3}(y+z)(x_1+x_2+x_3)^2\leq\frac1{27\sqrt3}L^3\,.
		\end{align}
    Here equality can be attained: it holds if and only if
		\[2y + 2z = 3x_1 = 3x_2 = 3x_3\quad\text{and}\quad 
			y = z = \dist(P,p_1) = \dist(P,q_4)\,,\]
		so that $N$ consists of parallel layers of honeycomb networks, 
    connected orthogonally.
	\end{caselist}
\end{proof}

We now discuss the other topology of irreducible networks of degree~$5$, 
namely the dipole graph~$D_5$ as the quotient.
Interestingly enough, like the double bouquet graph $D_{1,3}$,
also $D_5$ can be covered by connected parallel layers of hexagonal networks. 
However, the distances between these layers cannot be chosen 
as in Theorem~\ref{thm:bnn} because the four cycles generating the homology
lead to a different integer constraint. The \bnn\ network can be obtained
as a network covering the dipole graph $D_5$. Its quotient graph, however,
is always a covering graph of~$D_5$ with more than two vertices.
Hence another network arises as the optimal covering of~$D_5$, 
called the \sqp\ network:

\begin{theorem}\label{thm:sqp}
	Let $N$ be a triply periodic network with quotient $D_5$.  Then
	\begin{align}\label{eq:sqp}
		\frac{L^3}V\geq\frac{405}8=50.625\,.
	\end{align}
	In case of equality the five neighbours of each vertex 
  form the vertices of a square pyramid with height~$L/3$.
\end{theorem}

\begin{proof}
	Pick an arbitrary vertex $q_0\in N$ together with 
  its five neighbours $p_0,\ldots,p_4$.  
  Note that $L$ is the sum of the five edge lengths from $q_0$ to
  these points.
	\begin{figure}
		\centering
		\begin{subfigure}[t]{.45\linewidth}
			\centering
			\begin{tikzpicture}[x=3cm,y=3cm]
				\coordinate[dot1](p) at (0,0);
				\coordinate[dot2](q) at (1,0);
				\draw[steiner1](p) to [out=80,in=100] (q);
				\draw[steiner2](p) to [out=30,in=150] (q);
				\draw[steiner3](p) to [out=-30,in=210] (q);
				\draw[steiner4](p) to [out=-80,in=260] (q);
				\draw[steiner5](p) to (q);
				\node[right] at (q) {$q$};
				\node[left] at (p) {$p$};
			\end{tikzpicture}
		\end{subfigure}
		\begin{subfigure}[t]{.45\linewidth}
			\centering
			\begin{tikzpicture}[x=1.5cm,y=1.5cm]
				\coordinate[dot1](q0) at (0,0);
				\coordinate[dot2](p2) at (-.2,-1.2);
				\coordinate[dot2](p4) at (.5,.866);
				\coordinate[dot2](p3) at (.5,-.866);
				\coordinate[dot2](p1) at (-1,0);
				\coordinate[dot2](p0) at (-.3,1.2);
				\draw[steiner3](q0) -- node[above] {$x_1$} (p1);
				\draw[steiner1](q0) -- node[left] {$x_2$} (p2);
				\draw[steiner2](q0) -- node[right] {$x_3$} (p3);
				\draw[steiner5](q0) -- node[right] {$x_4$} (p4);
				\draw[steiner4](q0) -- node[left] {$x_0$} (p0);
				\node[right] at (q0) {$q_0$};
				\node[left] at (p1) {$p_1$};
				\node[right] at (p2) {$p_2$};
				\node[right] at (p3) {$p_3$};
				\node[left] at (p4) {$p_4$};
				\node[left] at (p0) {$p_0$};
			\end{tikzpicture}
		\end{subfigure}
		\caption{Topology and embedding of the dipole graph $D_5$.}
		\label{fig:sqp}
	\end{figure}
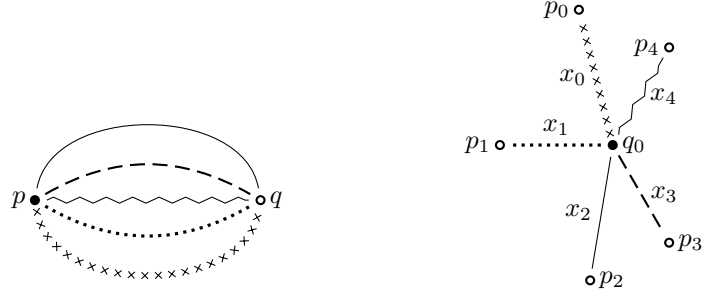

	We consider first the case that there exists a plane $P$ 
  which contains four of the neighbours~$p_i$.  
  Then the fifth neighbour cannot be contained in~$P$; 
  we suppose it is labelled~$p_0$.
  Moreover, we may assume the labelling is such that $p_4$ lies 
  on the lattice spanned by $p_1,p_2,p_3$.  
  Denote with $T$ the triangle in $P$ with vertices $p_1,p_2,p_3$.

  The convex hull of the four points $p_1$ to $p_4$ 
	is a triangle or a quadrilateral~$E\subset P$.
  By our assumption and the fact that $N$ is immersed, 
  its area satisfies $\area E\ge 2\area T$, 
  where equality corresponds to~$E$ being a parallelogramme.
  Denote by~$\Delta$ the pyramid with base~$E$ and apex~$p_0$.
  The volume~$V$ of a fundamental domain for the lattice 
  then is at most $3\vol\Delta$.
  As in Lemma~\ref{lemma:pyramid} we set
	$x_i:=\vert p_i-q_0\vert$ for $i=1,\ldots,5$, and $s:=x_1+x_2+x_3+x_4$.  
  The volume estimate and the lemma give
	\begin{align}\label{eq:sqppyramid}
		\frac{L^3}V
		\geq\frac{L^3}{3\operatorname{vol}\Delta}
		\geq\frac3{\area E}\Bigl(\frac{45}{32}\,s\Bigr)^2.
	\end{align}

  Equality in \eqref{eq:sqppyramid} is equivalent to both inequalities 
  attaining equality.  
  The first inequality holds with equality 
  if $\area E=2\area T$ so that $E$ is a parallelogramme.  
  Lemma~\ref{lemma:pyramid} characterizes the case that 
  the second inequality holds with equality:
  We must have 
  \begin{equation}\label{equcase}
    x_1=x_2=x_3=x_4=\frac8{13}x_0\qquad\text{and}\qquad \dist(q_0,P)=\frac14x_1,
  \end{equation}
  as well as $p_0-q_0$ perpendicular to $P$.
  Since \eqref{equcase} implies that
  $p_1$ to $p_4$ are contained in a circle in $P$, 
  the parallelogramme must be a rectangle,
  and moreover $p_0$, $q_0$ project orthogonally onto its midpoint, 
  having distances from $P$ prescribed by~\eqref{equcase}.

  Among the equality cases, \eqref{eq:sqppyramid} attains
  its minimum when the right hand side is minimal; 
  moreover, this establishes a valid lower bound for the length quotient $L^3/V$.
  The only freedom is the conformal parameter of the rectangle. 
  Clearly, minimality of \eqref{eq:sqppyramid} occurs for maximal $\area E$, 
  i.e., for $E$ a square.  To compute \eqref{eq:sqppyramid} for
  this case note the diagonal of~$E$ has a length~$c$ satisfying
  \[
    \Bigl(\frac c2\Bigr)^2=x_1^2-\Bigl(\frac{x_1}4\Bigr)^2 
    = \frac{15}{16}\,x_1^2,\quad\text{ and so }\quad
    \area E = 2\Bigl(\frac c2\Bigr)^2 
    = \frac{15}8\, x_1^2.
  \] 
	Inserting this expression into \eqref{eq:sqppyramid}, thereby using $s=4x_1$, 
  gives the desired estimate~\eqref{eq:sqp}
	and verifies the claims for the equality case.
  %

	\begin{figure}
		\centering
		\begin{subfigure}[t]{.48\linewidth}
			\centering
			\begin{tikzpicture}[x=2.8cm,y=2.8cm] 
				\coordinate[dot2](p0) at (0,0);
				\coordinate[dot2](p3) at (1,1);
				\coordinate[dot2](p2) at (.5,1.2);
				\coordinate[dot2](p1) at (-.5,1);
				\coordinate[dot2](p4) at ($(p1)+(p2)$);
				\coordinate[dot1](q0) at (.2,.8);
				\draw[dotted,thick,gray,->](p1)--(p4);
				\draw[dotted,thick,gray,->](p2)--(p4);
				\draw[dotted,thick,gray,->](p3)--(p4);
				\draw(q0)--(p0);
				\draw(q0)--(p1);
				\draw(q0)--(p2);
				\draw(q0)--(p3);
				\draw(q0)--(p4);
				\draw[fill=gray,gray,fill opacity=.2](1,1)--(.5,1.2)--(-.5,1)--(1,1);
				\draw[->,thick,gray](p0)--node[left]{$g_1$}(p1);
				\draw[->,thick,gray](p0)--node[right]{$g_2$}(p2);
				\draw[->,thick,gray](p0)--node[right]{$g_3$}(p3);
				\node[below] at (p0) {$p_0=0$};
				\node[left] at (p1) {$p_1$};
				\node[above right] at (p2) {$p_2$};
				\node[right] at (p3) {$p_3$};
				\node[right] at (p4) {$p_4$};
				\node[below left] at (q0) {$q_0$};
			\end{tikzpicture}
		\end{subfigure}
		\caption{Notation for Theorem~\ref{thm:sqp}.  
      Shown is a case where the vertices $p_0,\ldots,p_4$ do not form a pyramid.}
		\label{fig:sqpproof}
	\end{figure}
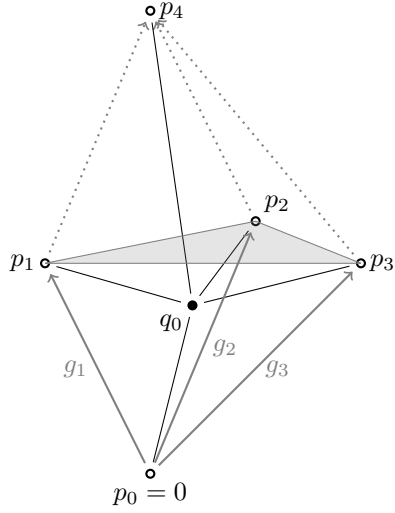

	Now suppose no four $p_i$'s are coplanar.
	We may assume that $p_0$ is the origin and the indexing is such that
  the lattice $\Lambda$ is spanned by $g_i := p_i - p_0$ for $i=1,2,3$, 
	see Figure~\ref{fig:sqpproof}. 
  Then $p_4$ is a lattice vector and so there are integer coefficients 
	$\lambda_1,\lambda_2,\lambda_3\in\mathbb Z$ such that
	\begin{align}\label{eq:sqpcoefficients}
		p_4 = \lambda_1p_1+\lambda_2p_2+\lambda_3p_3\,.
	\end{align}
	Let $P_{123}$ be the plane through $p_1,p_2,p_3$ and consider the vector
	\[n_{123} := (p_2-p_1)\times(p_3-p_1)\]
	normal to $P_{123}$. 
	The point $p_4$ has a signed distance from $P_{123}$ given by
	$\sdist(p_4,P_{123}) =  \langle n_{123}/\vert n_{123}\vert,p_4-p_1\rangle$.
  Rewriting \eqref{eq:sqpcoefficients} as
	\[p_4-p_1=(\lambda_1+\lambda_2+\lambda_3-1)p_1 + \lambda_2(p_2-p_1)+\lambda_3(p_3-p_1)\,,\]
  we see the signed distance is
	\begin{align}\label{eq:sqp123}
		\sdist(p_4,P_{123}) 
		= (\lambda_1+\lambda_2+\lambda_3-1)\mathcal V\,,
	\end{align}
	where $\mathcal V = \langle p_1,p_2\times p_3\rangle=\det(p_1,p_2,p_3)$ 
	is a signed volume of~$N/\Lambda$. 
	After relabeling we may assume that $p_0$ and~$p_4$ lie on different sides 
	of~$P_{123}$, so that $\lambda_1+\lambda_2+\lambda_3\geq1$. 
  Since no four $p_i$'s are coplanar, 
	in fact $\lambda_1+\lambda_2+\lambda_3\geq2$ 
  and $\lambda_i\neq0$ for $i=1,2,3$.
  Moreover, we may assume $p_1,p_2,p_3$ are indexed 
	such that $\lambda_1\leq\lambda_2\leq\lambda_3$. 

	We now distinguish four cases for $(\lambda_1,\lambda_2,\lambda_3)$.
	In all cases there is a plane~$P$ through three of the~$p_i$ 
	such that the remaining two vertices lie to opposite sides of~$P$ 
  at different distances.
  In all cases, the result will be lower bound on $L^3/V$
  which is strictly larger than~\eqref{eq:sqp}. 
	\begin{caselist}
		\item\label{item:sqplarge} Suppose $\lambda_1+\lambda_2+\lambda_3\geq3$. 
		Equivalently, by \eqref{eq:sqp123}, we have
		\[2\dist(p_0,P_{123})\leq\dist(p_4,P_{123})\,.\]
		Thus, if $A_{123}$ denotes the area of the triangle with 
    vertices $p_1,p_2,p_3$ we find
		\begin{align}\label{eq:sqphcbvolume}
			V = 2A_{123}\dist(p_0,P_{123}) 
			&\leq2A_{123}\frac{\dist(p_0,P_{123})+\dist(p_4,P_{123})}3\,.
		\end{align}
		We set $x_i := \vert p_i-q_0\vert$ for $i=0,\ldots,4$, and claim
		\begin{align}\label{eq:sqpmaincase}
			A_{123}\leq \frac1{\sqrt3}\Bigl(\frac{x_1+x_2+x_3}2\Bigr)^2 \,.
		\end{align}
		
		To verify the claim, assume~$q_0$ minimizes $x_1+x_2+x_3$. 
    If $q_0$ coincides with $p_3$, 
    then the estimate on geometric and arithmetic mean gives
		\[A_{123}\leq\frac12x_1x_2\leq\frac12\Bigl(\frac{x_1+x_2+x_3}2\Bigr)^2 \,,\]
		thus proving~\eqref{eq:sqpmaincase}. 
    The same reasoning leads to \eqref{eq:sqpmaincase} if $q_0=p_1$ or~$q_0=p_2$. 
    If, however, $q_0\notin\{p_1,p_2,p_3\}$, then the network is balanced at~$q_0$.  
    Choosing coordinates as in \eqref{eq:bnncoordinates} 
    leads to estimate~\eqref{eq:bnnvertical}. This proves the claim.
		
		Inserting~\eqref{eq:sqpmaincase} into \eqref{eq:sqphcbvolume} gives
		\begin{align*}
			V\leq\frac2{3\sqrt3}\Bigl(\frac{x_1+x_2+x_3}2\Bigr)^2(x_0+x_4)
      \leq\frac2{81\sqrt 3}L^3\,.
		\end{align*}
		This verifies \eqref{eq:sqp}.
		\item Suppose $\lambda_1\leq-2$. 
		We consider the plane $P_{023}$ spanned by $p_0,p_2,p_3$ 
		with normal vector $n_{023}:=p_2\times p_3$. 
		Using \eqref{eq:sqpcoefficients}, we have
		\[\vert n_{023}\vert\sdist(p_1,P_{023})=\langle n_{023},p_1\rangle = \mathcal V\,,\qquad\vert n_{023}\vert\sdist(p_4,P_{023})=\langle n_{023},p_4\rangle=\lambda_1\mathcal V\,.\]
		Since $\lambda_1\leq-2$, 
		the vertices $p_1$ and $p_4$ lie on opposite sides of $P_{023}$ and
		\[2\dist(p_1,P_{023})\leq\dist(p_4,P_{023})\,.\]
		As in \ref{item:sqplarge} 
		we obtain again $L^3/V\ge 81\sqrt 3/2$. 
		\item Suppose $\lambda_3\geq3$. Then, by \eqref{eq:sqpcoefficients}, 
		\[\vert\det(p_1,p_2,p_4)\vert = \lambda_3\vert\det(p_1,p_2,p_3)\vert\geq 3V\,.\]
		Applying the estimate \eqref{eq:dia} on the length of a network covering~$D_4$
		to the subnetwork spanned by the four edges from $q_0$ to $p_0,p_1,p_2,p_4$ 
    shows again that \eqref{eq:sqp} holds strictly:
		\[L^3>(x_0+x_1+x_2+x_4)^3   
      \geq12\sqrt3\,\vert\det(p_1,p_2,p_4)\vert 
			\geq36\sqrt3\,V\,.        
    \]
		\item Finally, assume $\lambda_1=-1$, $\lambda_2=1$ and $\lambda_3=2$. In this case we consider the plane $P_{014}$ spanned by $p_0,p_1,p_4$ with normal vector $n_{014}:=p_1\times p_4$. Then, by \eqref{eq:sqpcoefficients},
		\[\vert n_{014}\vert\sdist(p_2,P_{014})=\langle n_{014},p_2\rangle=-2\mathcal V,\qquad\vert n_{014}\vert\sdist(p_3,P_{014})=\langle n_{014},p_3\rangle=\mathcal V\,.\]
		So the vertices $p_2$ and $p_3$ lie on opposite sides of $P_{014}$, and
		\[2\dist(p_3,P_{014})\leq\dist(p_2,P_{014})\,.\]
		After relabelling the $p_i$ we proceed again as in~\ref{item:sqplarge}.
	\end{caselist}
	A moment's thought gives that the four cases cover all admissible values for $\lambda_1,\lambda_2,\lambda_3$, and so \eqref{eq:bnn} holds strictly when no four
  $p_i$ are coplanar.
\end{proof}
The length quotient for irreducible triply periodic networks of a
degree higher than $6$ must be larger than the value obtained for
the two irreducible networks of degree~$5$. 

\begin{corollary}\label{cor:highdegthree}
	Let $N$ be an irreducible triply periodic network of degree $d\geq7$. Then
	\begin{align}\label{eq:highdegthree}
		\frac{L^3}V>\frac{405}8\,.
	\end{align}
\end{corollary}
\noindent
Thus for dimension $n=3$ the length quotient 
of networks with degree $d\geq7$ is always larger than the quotient for
all explicitly discussed cases with degree $3$ to~$6$.
\begin{proof}
	For even $d\geq8$ the quotient network $N/\Lambda$ covers the bouquet graph $B_{d/2}$. 
  Then~\eqref{eq:highdegthree} follows immediately from Theorem~\ref{th:pcu}, as
	\begin{align}\label{eq:highdegthreeeven}
		\frac{L^3}V\geq\Bigl(\frac82-3+1\Bigr)\,3^3=54\,.
	\end{align}

  For odd degree~$d$ the quotient network $N/\Lambda$ is classified by Proposition~\ref{prop:structure}:
  It covers the double bouquet graph $D_{\ell,k}$ with $k\geq3$ and $\ell\geq0$.
	Assume first the number of loops in $D_{\ell,k}$ which lift to generators of the lattice $\Lambda$ is at least $3$. Then $\ell\geq2$ and $N/\Lambda$ contains a (possibly disconnected) subgraph $N'/\Lambda$ which consists of four closed geodesics in $\R^3/\Lambda$, three of which lift to generators of~$\Lambda$. 
Note that the length of a closed geodesic 
is invariant under translation. So we may estimate the length of $N'/\Lambda$ by a network where the four geodesics intersect at one vertex. The reasoning of 
the proof of Theorem~\ref{th:pcu} then yields~\eqref{eq:highdegthreeeven}. 
	
	Now suppose that the loops in $D_{\ell,k}$ lift to at most two generators of $\Lambda$. 
  If exactly two loops lift to generators of $\Lambda$, 
  then if necessary we reason as before to assume that each lift is based 
  at a different vertex of~$N/\Lambda$. 
  Thus in any case $N$ contains a subnetwork $N'\subsetneq N$ 
  which is triply periodic and covers~$D_{1,3}$ or~$D_5$. 
  We conclude the length quotient of $N'$ is estimated by Theorem~\ref{thm:bnn} 
  or~\ref{thm:sqp}, and $N$ has a strictly larger quotient, as desired.
\end{proof}

\nocite{*}
\bibliographystyle{amsplain}
\bibliography{degree}

\end{document}